\documentclass[12pt,reqno]{amsart}

\usepackage{setspace}
\usepackage{color}
\usepackage[all]{xy}
\usepackage{amssymb,amsmath, amsfonts, amsthm}

\usepackage{enumitem}

\calclayout

\usepackage[latin5]{inputenc}

\newcommand{\Qq}{\mathbb Q}
\newcommand{\Zz}{\mathbb Z}
\newcommand{\Ff}{\mathbb F}

\newcommand{\squ}{\preceq}
\newcommand{\nsub}{\trianglelefteq}
\newcommand{\nor}{\trianglelefteq}

\newcommand{\ext}{{\rm Ext}}
\newcommand{\des}{{\rm Des}}
\newcommand{\krn}{{\rm Ker}}
\newcommand{\res}{{\rm Res}}
\newcommand{\defres}{{\rm Defres}}

\def\obs{{\rm Obs}}
\def\lin{{\rm Lin}}
\def\rk{{\rm rk}}

\def\ind{{\rm Ind}}
\def\hom{{\rm Hom}}
\def\aut{{\rm Aut}}
\def\coind{{\rm Coind}}
\def\rc{{\rm c}}
\def\defl{{\rm Def}}
\def\infl{{\rm Inf}}
\def\iso{{\rm Iso}}
\def\St{{\rm St}}

\def\cB{\mathcal{B}}
\def\cD{\mathcal{D}}
\def\cA{\mathcal{A}}
\def\cE{\mathcal{E}}
\def\cF{\mathcal{F}}
\def\cS{\mathcal{S}}
\def\cO{\mathcal{O}}
\def\cC{\mathcal{C}}
\def\cX{\mathcal{X}}
\def\cG{\mathcal{G}}

\newtheorem{thm}{Theorem}[section]
\newtheorem{pro}[thm]{Proposition}
\newtheorem{cor}[thm]{Corollary}
\newtheorem{lem}[thm]{Lemma}

\theoremstyle{definition}
\newtheorem{rem}[thm]{Remark}

\newtheorem{defn}[thm]{Definition}
\newtheorem{ex}[thm]{Example}

\def\maprt#1{\smash{\,\mathop{\longrightarrow}\limits^{#1}\,}}
\newcommand{\leftexp}[2]{{\vphantom{#2}}^{ #1}{\hskip-1pt#2}}

\begin{document}

\title{Obstructions for gluing biset functors}
\author[O. Co\c skun]{Olcay Co\c skun}
\address{Department of Mathematics, Bo\u gazi\c ci University, 
34342 Bebek, \. Istanbul, Turkey}
\email{olcay.coskun@boun.edu.tr}
\author[E. Yal\c{c}{\i}n]{Erg\"un Yal\c{c}{\i}n}
\address{Department of Mathematics, Bilkent University,
 06800 Bilkent, Ankara, Turkey}
\email{yalcine@fen.bilkent.edu.tr }
\subjclass[2010]{} 
 
 \keywords{}

\date{\today}

\begin{abstract}  We develop an obstruction theory for the existence and uniqueness of a solution to the gluing problem
for a biset functor defined on the subquotients of a finite group $G$.  The obstruction groups for this 
theory are the reduced cohomology groups of a category $\cD^*_G$ whose objects are the sections $(U,V)$ 
of $G$, where $1\neq V\trianglelefteq U \leq G$, and whose morphisms are defined as a generalization of morphisms in the orbit 
category. Using this obstruction theory, we calculate the obstruction group for some well-known $p$-biset functors, 
such as the Dade group functor defined on $p$-groups with $p$ odd.
\end{abstract}

\maketitle


\section{Introduction and definitions}


Given a commutative ring $R$ with unity, the $R$-biset category 
$R\cB$ is a category whose objects are all finite groups, and in which the morphisms from the finite 
group $H$ to the finite group $K$ are given by the Burnside module $R \otimes _{\mathbb Z} B(K,H)$ 
of $(K,H)$-bisets, or equivalently of left $K\times H$-sets. The composition in the biset category is given by the composition product of bisets 
(see Section \ref{sect:DestOrbit} for details).  An $R$-linear functor $$F: R \cB \rightarrow R\text{-mod}$$ is 
called a \emph{biset functor} over $R$. The Burnside ring 
functor $B$ and the representation ring functor $R_K$ defined on finite groups, and the Dade group 
$D$ defined on $p$-groups with $p$ odd, are important examples of biset functors.   

Let $G$ be a finite group. A pair $(U,V)$ of subgroups of $G$ is called a \emph{section} 
of $G$ if $V \nsub U \leq G$. The quotient group $U/V$ is called a \emph{subquotient} of $G$. 
In the literature sometimes these two terms are used in an equivalent way. Our convention will be as follows: 
Let $\cX ^s$ denote the 
set of all quotient groups $U/V$ indexed by the set of sections $(U,V)$ of $G$. We denote the $R$-biset category 
over the collection $\cX^s$ by $R\cB _G$. An $R$-linear functor 
$F: R \cB_G \rightarrow R$-mod is called a \emph{biset functor for $G$ over $R$}.

Every $(K, H)$-biset can be expressed as a composition of five types of bisets, called induction, inflation,
isogation, deflation, and restriction bisets. Descriptions of these bisets can be found in \cite[2.3.9]{Bouc-Book} 
(isogation is a biset associated to an isomorphism, see  \cite[pg. 3811]{Barker-Rhetorical}). 
For $V \nsub U \leq G$, the composition of deflation and restriction maps  $\res^G _U$ and $\defl ^U _{U/V}$ 
is often denoted by $\defres^G _{U/V}$ or $\des^G _{U/V}$, and is called the \emph{destriction biset}. 
Similar to biset functors, one can also define a notion of \emph{destriction functor}, which uses only the 
destriction maps and isogations induced by conjugations. Every biset functor can be considered as a destriction functor 
via the algebra map from the destriction algebra to the alchemic algebra over the collection $\cX^s$ (see Section \ref{sect:destriction} for details).
 
For every $K \nor H \leq G$, the normalizer of the section $(H,K)$ is defined 
as the subgroup $N_G(H, K)=N_G(H) \cap N_G(K)$. 

\begin{defn}\label{def:GluingProblem} Let $F$ be a biset functor or a destriction functor for $G$ over $R$. 
A {\bf gluing data} for $F(G)$ is 
a sequence $(f_H)_{1<H \le G}$ of elements $f_H \in F(N_G(H)/H)$ satisfying the following 
compatibility conditions:
\begin{enumerate}
  \item[(i)] (Conjugation invariance) For any $g\in G$ and $H \le G$, we have $$\leftexp{g} f_H = f_{\leftexp{g} H}.$$
  \item[(ii)] (Destriction invariance) For any $K \nor H \leq G$, we have 
  $$\defres^{N_G(K)/K}_{N_G(H, K)/H }f_K = \res^{N_G(H)/H }_{N_G(H, K)/H}f_H. $$
\end{enumerate}

The {\bf gluing problem} for the biset functor $F$ at $G$ is the problem of finding an element $f\in F(G)$ 
such that for any non-trivial subgroup $H$ of $G$, we have $\defres^G_{N_G(H)/H}f = f_H.$ If such an 
element exists, we call it a solution to the given gluing data. 
\end{defn}

The gluing problem was first introduced for endo-permutation modules by Bouc and Th\'{e}venaz 
\cite{BoucThev-Glue}.  When a solution to a gluing problem exists we would like to also know if it is unique. 
Following Bouc and Th\'{e}venaz \cite{BoucThev-Glue}, we denote the set of all gluing data for $F(G)$ 
as the inverse limit  $$\varprojlim_{1<H\le G}F(N_G(H)/H ).$$ Later we show that the gluing data is indeed 
an inverse limit over a category. The complete solution to the gluing problem is an exact sequence
$$
\xymatrix{0 \ar[r] & \krn (r^F_G) \ar[r] & F(G) \ar[r]^-{r^F_G} & \varprojlim \limits_{1<H\le
G}F(N_G(H)/H)  \ar[r] &  \obs(F(G)) \ar[r] & 0 } 
$$
where $r^F_G$ denotes the map which takes $f \in F(G)$ to the tuple $(\defres^G_{N_G(H)/H} f)_{1< H \leq G}$
and  $\obs (F(G))$ denotes the group of obstructions for a gluing data to have a solution. 
 
Bouc and Th\' evenaz \cite[Theorem 1.1]{BoucThev-Glue} calculated the obstruction group for the torsion 
part of the Dade group $D_t (G)$ when $G$ is a noncyclic $p$-group with $p$ odd. They found that
$$ \obs(D_t (G))\cong  H^0 (\cA_{\geq 2} (G) ; \Ff _2 ) ^G$$
where $\cA _{\geq 2} (G)$ is the poset of elementary abelian subgroups of $G$ with rank $\geq 2$. Later 
Bouc \cite[Theorem 2.15]{Bouc-Glue} showed that the obstruction group $\obs(D(G))$ for the Dade group 
$D(G)$ of a $p$-group $G$, with $p$ odd, embeds into the group 
$H^1 (\cA_{\geq 2} (G) ; \Zz) ^{(G)}$ defined in \cite[Notation 2.9]{Bouc-Glue} using $G$-invariant cocycles. 
We show in Lemma \ref{lem:CohIsom}
that this group is isomorphic to the first cohomology 
group $H^1 ( \cA_{\geq 2} (G) /G ; \Zz )$
of the orbit space $\cA_{\geq 2} (G)/G$.

Co\c skun \cite{Coskun-GlueBS} extended these calculations to the biset functor for the dual of the rational 
representation ring $R_{\Qq}^*$ and to the biset functor of Borel-Smith functions $C_b$. In all these calculations 
the results are expressed as modified versions of  cohomology groups for some topological space. Our aim in this 
paper is to give a common framework for all these obstruction group calculations by expressing the obstruction group 
for a biset functor $F$ as the 0-th reduced cohomology group of a category $\cD_G$, which we introduce below. 
Using this theory we also obtain new results for the obstruction groups of certain $p$-biset functors  such as 
the Dade group functor.
 
The set of sections of $G$ is a $G$-poset under the conjugation action  $(U,V) \to ( \leftexp{g} U, \leftexp{g}V)$ 
of $G$, where  $\leftexp{g}U$ denotes the subgroup $gUg^{-1}$. The order relation is given by the inclusion of sections,
that is, $(U,V) \squ (M,L)$ if and only if $L\le V\le U\le M$. We call a set of sections of $G$ a \emph{collection}
if it is closed under conjugation.

\begin{defn}\label{def:OrbCatSec} Let $\mathcal C$ be a collection of sections of a finite group $G$. The \textbf{orbit category of sections} 
$\cD^{\mathcal C} _G$ is the category whose objects are sections $(U, V)$ in $\mathcal C$, and whose  morphisms 
are given by  $$\hom _{\cD_G^{\mathcal C} } ((U,V) , (M,L))=\{ Mg \ | \  g\in G, \ \leftexp{g} (U, V) \squ (M, L) \}.$$
If $Nh: (M, L) \to (N, R)$ is a morphism in $\cD_G^\mathcal C$, then the composition is given by $Nh\circ Mg =Nhg$.  
We denote the orbit category of sections over all sections of $G$ simply by $\cD_G$.
\end{defn}
 
Note that the composition defined above is well-defined: If $g'=xg$ and $h'=yh$ for some $x\in M$ and $y\in N$, then
$Nh'g'=Nhxg=Nhg$ because $hxh^{-1} \in \leftexp{h} M \leq N$. We show that the category algebra for the opposite category of
category $\cD_G ^{\cC}$ is isomorphic to destriction algebra (see Definition \ref{pro:DesAlg} and Proposition \ref{pro:Equivalence}).
Using the algebra map from the destriction algebra to the alchemic 
algebra, we can consider a biset functor $F$ as an $R\cD^{\cC}_G$-module and study the gluing problem for $F$ using homological 
algebra over $R\cD^{\cC}_G$-modules. To simplify the notation we write $F(U/V)$ for $F\bigl ( (U,V) \bigr ) $ when $F$ 
is an $R\cD_G$-module. We also write $F(G)$ for $F(G/1)$ to simplify the notation further.

Given a collection of sections $\cC$ of $G$, closed under taking subsections, let $T^{\cC}$ denote the 
$R\cD_G$-module whose values on those sections $(U, V)\not \in \cC$ are equal to $R$, and equal 
to zero on the sections in $\cC$. Let $\underline{R}$ denote the constant functor of $\cD_G$ and $J^{\cC}$ denote 
the $R\cD _G$-module such that $0 \to J^{\cC} \to \underline{R} \to T^{\cC} \to 0$ is an exact sequence.  
For a destriction functor or a biset functor $F$, this gives a long exact sequence  
\[
\xymatrix{0 \ar[r] & \hom_{R\cD_G} (T^{\cC}, F)  \ar[r] & F(G) \ar[r]^-{r^F_G} & \varprojlim \limits _{(U,V) \in 
\cD^{\cC}_G} F(U/V)  \ar[r] & \ext{}^1_{R\cD_G} (T ^{\cC} ,F)  \ar[r] & 0},
\] 
which can be regarded as a general version of a gluing problem for a collection  $\cC$ that is closed under taking 
subsections. The obstruction group for this gluing problem over $\cC$ is given by 
$\obs ^{\cC} (F) \cong \ext{}^1_{R\cD_G} (T ^{\cC} ,F)$.

We show that the gluing problem for $F$ stated in Definition \ref{def:GluingProblem} is a special case of 
this more general gluing problem over the category $\cD_G^{\cC}$. 

\begin{thm}\label{thm:Main1} 
Let $F$ be a destriction functor or a biset functor for  a finite group $G$ defined over $R$. Let $\cD_G^*$ denote 
the orbit category of sections over the collection of all sections $(U,V)$ of $G$ with $V\neq 1$, and  
$T$ the constant functor on the set of all sections $(U,V)$ with $V=1$. 
Then there is an isomorphism 
$$\varprojlim \limits_{1<H\le
G}F(N_G(H)/H) \cong \varprojlim \limits _{(U,V) \in \cD^*_G} F(U/V),$$ which gives isomorphisms
\[\obs(F(G)) \cong \ext_{R\cD_G} ^1 (T , F)   \text{         and         } \ 
\krn (r^F_G) \cong \hom_{R\cD_G}(T,F).\] 
\end{thm}

This theorem is proved in Section \ref{sect:ObsGluing}  as Propositions \ref{pro:isom2}  and \ref{pro:main1}. 
In Section \ref{sect:Reduction}, we discuss the gluing problem for an arbitrary collection $\cC$ (not necessarily 
closed under taking subsections). We define the reduced cohomology $\widetilde H^* (\cD_G^{\cC} ; F) $ of 
the category $\cD_G^{\cC}$ and show that if the collection $\cC$ is closed under taking subsections, then 
$$\obs^{\cC} ( F(G)) \cong \ext^1 _{\cD_G } (T^{\cC} , F) \cong \widetilde H^0 (\cD_G ^{\cC} ; F).$$

To calculate these reduced cohomology groups, we apply a theorem of Jackowski and Slominska \cite{JackSlom} 
on isotropy presheaves to the category $\cD_G^{\cC}$ and prove that under certain conditions the collection $\cC$ 
can be replaced by a smaller collection $\cC'$ with isomorphic reduced cohomology group. When $G$ is a $p$-group, 
we show that the collection of all sections $(U,V)$ of $G$ with $V\neq 1$ can be reduced to the collections of all sections 
of the form $(C_G(E), E)$ where $E$ is a nontrivial elementary abelian subgroup of $G$. 

This reduction allows us to relate the obstruction groups $\obs (F(G))$ to the cohomology groups of the Quillen category 
$\cA_p (G)$, the category whose objects are the nontrivial elementary abelian $p$-subgroups of $G$ and morphisms
$E_1 \to E_2$ are given by compositions of conjugation and inclusion maps. Given a destriction or biset functor $F$, there is an associated 
$R\cA_p(G)$-module $\overline F$ defined by $\overline F (E)=F(C_G(E)/E)$ for every $E\in \cA_p (G)$. 
We prove the following theorem.  

\begin{thm}\label{thm:Main2} 
Let $G$ be a $p$-group and $\cA_p (G)$ denote the Quillen category of $G$ over all  nontrivial elementary abelian 
$p$-subgroups of $G$. Then, for a destriction functor or a biset functor $F$, there is an isomorphism 
$$\widetilde H^* ( \cD^* _G  ; F) \cong \widetilde H^* (\cA_p (G) ; \overline F ).$$  
In particular,  $\obs(F (G)) \cong \widetilde H^0 (\cA_p (G) ; \overline F)$.
\end{thm}

As a consequence of Theorem \ref{thm:Main2} and a theorem of Oliver \cite[Theorem 1]{Oliver}, we also 
obtain that if $G$ is a $p$-group then $$\ext_{R\cD_G} ^n (T , F) \cong \widetilde H^{n-1} (\cD_G ^* ; F)=0$$ for $n \geq \rk (G)+1$, where 
$\rk (G)$ denotes the maximal rank of elementary abelian subgroups in $G$ (see Proposition \ref{pro:Oliver}). 
 
In Section \ref{sect:Rhetorical} we consider rhetorical $p$-biset functors and give an explicit formula for the 
obstruction groups $\obs (F(G))$ for these functors (see Theorem \ref{thm:Rhetorical}). In Section \ref{sect:GluingEndo}, 
we apply this formula to gluing problems for various rhetorical $p$-biset functors. For the torsion part of the Dade group 
$D_t$ defined on $p$-groups with $p$ odd, we recover the earlier obstruction group computations done by Bouc 
and Th\'{e}venaz \cite{BoucThev-Glue}. We also recover the obstruction group calculations done 
by Co\c skun \cite{Coskun-GlueBS} for the dual of the rational representation ring functor $R_{\Qq}^*$ and for the 
Borel-Smith functor $C_b$ (Propositions \ref{pro:ObsR_Q^*} and \ref{pro:ObsCb}). We also calculate the obstruction 
group $\obs(D_t(G))$ for a $2$-group $G$ (see Proposition \ref{pro:ObsDt}).

In Section \ref{sect:GluingEndo} we consider the obstruction group for the Dade group functor $D$ defined over 
$p$-groups with $p$ odd. We first prove a vanishing result  $H^1 (\cD^*_G; B^*)=0$ for the dual of 
the Burnside ring functor $B^*$ using a base change spectral sequence for ext-groups (see Lemma 
\ref{lem:Vanish}). As a consequence of this vanishing result  we prove the following theorem.

\begin{thm}\label{thm:Main3} 
Let $G$ be a $p$-group with $p$ odd. Let $D$ denote the $p$-biset functor for the Dade group. 
Then there is an exact sequence  of abelian groups
$$ 0 \to \obs(D(G)) \to H^1 (\cA_{\geq 2} (G) /G ; \Zz) \to H^1 (\cA_{\geq 2} (G)/G; \Zz/2),$$
where the second map is induced by the mod $2$ reduction map $\Zz \to \Zz/2$.
\end{thm}

This theorem is proved as Theorem \ref{thm:ObsD} in Section \ref{sect:GluingEndo}. One of the key 
ingredients is the computation of $H^1(\cD_G^*; F)$ for a rhetorical $p$-biset functor satisfying certain 
properties (see Proposition \ref{pro:H1Calc}).

The paper is organized as follows: We introduce the destriction algebra and the orbit category of sections 
in Section \ref{sect:DestOrbit}. The obstruction groups for the gluing problem are defined in Section \ref{sect:ObsGluing}. 
Theorem \ref{thm:Main1} is also proved in this section. In Section \ref{sect:Reduction} we define the reduced 
cohomology groups for the category $\cD^{\cC}_G$, and apply a theorem of Jackowski and S{\l}ominska to these 
reduced cohomology groups. In Section \ref{sect:Quillen}  we show that the reduced cohomology groups of 
$\cD^{\cC}_G$ can be calculated as the cohomology groups of the Quillen category $\cA_p(G)$ when $G$ is 
a $p$-group (Theorem \ref{thm:Main2}). In Section \ref{sect:Rhetorical} we prove Theorem \ref{thm:Rhetorical}, 
which gives an explicit description of the obstruction group $\obs(F(G))$ for a rhetorical $p$-biset functor $F$. 
Finally in Section \ref{sect:GluingEndo} we apply Theorem \ref{thm:Rhetorical} to calculate the obstruction 
groups for some rhetorical $p$-biset functors related to endo-permutation modules. Theorem \ref{thm:Main3} 
is proved in this final section.

\vskip 5pt
 
{\bf Acknowledgement:}  The second author is supported by a T\" ubitak 1001 project (grant no: 116F194).   
We thank the referee for careful reading of the paper and for many helpful suggestions that improved the paper.


\section{Destriction algebra and the orbit category of sections}\label{sect:DestOrbit}

In this section, we introduce basic definitions about biset functors and destriction algebra. Then we define 
the orbit category of sections $\cD^{\cC}_G$ over a collection of sections $\cC$ and prove that the module 
category over the orbit category of sections is equivalent to the module category over the destriction algebra.
 
\subsection{Biset functors}\label{sect:Bisets} Let $K$ and $H$ be finite groups.  A  $(K,H)$-biset is a finite 
set $X$ together with a left $K$-action and a right $H$-action such that $k(xh)=(kx)h$ for every $k\in K$, 
$h\in H$, and $x\in X$. A $(K,H)$-biset $X$ can be considered as a left $K\times H$-set with the action 
given by $(k,h)x=kxh^{-1}$. We define the Burnside group of $(K,H)$-bisets $B(K,H)$ as the Burnside group 
$B(K\times H)$ of left $K\times H$-sets. Given another finite group $L$, we define a composition product
\[
\text{--} \,\times_H \text{--} : B(K,H)\times B(H,L)\rightarrow B(K,L)
\]
as the linear extension of the correspondence $(X,Y)\mapsto X\times_H Y$ where $X\times_H Y = (X\times Y)/H$ 
is the set of $H$-orbits of the set $X\times Y$ under the action $h\cdot (x,y):= (xh^{-1}, hy)$.

Let $R$ be a commutative ring with unity. The \emph{R-biset category} $R\cB$ is a category whose objects are all finite 
groups, and in which the morphisms from the finite group $H$ to the finite group $K$ are given by the Burnside 
module $R \otimes _{\mathbb Z} B(K,H)$ of $(K,H)$-bisets. The composition in the biset category is given by 
the above composition product. For any finite group $H$, the identity morphism of $H$ in $R\cB$ is equal to 
$R \otimes {\rm id} _H$, where ${\rm id}_H={}_H H_H$ is the identity biset at $H$. 

Recall that if $U, V$ are subgroups of a finite group $G$ such that $V \nsub U$, then the quotient group $U/V$ is called a 
subquotient of $G$, and the pair $(U,V)$ is called a section of the group $G$.  
Let $\cX$ be a finite set of finite groups closed under taking subquotients up to isomorphism. By this we mean that 
if $G \in \cX$, then a subquotient $U/V$ of $G$ is isomorphic to a group in $\cX$.

\begin{defn}\label{def:biset} 
The $R$-biset category $R\cB^{\cX}$ over the collection $\cX$ is the full subcategory of the biset category $R\cB$ 
whose objects are groups in $\cX$. An $R$-linear functor $$F: R \cB ^{\cX} \rightarrow R\text{-mod}$$ 
is called a \emph{biset functor} defined on the collection $\cX$ over $R$.
\end{defn}

To define a biset functor for a fixed finite group $G$, we take $\cX$ as 
the collection $\cX^s$ which is defined as the set of all finite groups indexed by sections $(U,V)$ of $G$, 
where the group in $\cX^s$ corresponding to the section $(U,V)$ is the quotient group $U/V$.
In this case we denote the corresponding $R$-biset category by $R\cB_G$, and an $R$-linear functor $F: R \cB_G \rightarrow R$-mod 
is called a {\bf biset functor} for $G$ over $R$. 

We also recall the definition of the alchemic algebra for an arbitrary collection $\cX$.

\begin{defn}\label{def:alchemic} 
The \emph{alchemic algebra} for $\cX$ over $R$ is the $R$-algebra 
$$R\Gamma ^{\cX} =\bigoplus _{K,H \in \cX} RB(K,H)$$
with multiplication given by biset composition products. A biset functor $F$ defined on $\cX$ over $R$ is 
defined to be an $R\Gamma ^{\cX}$-module (see \cite{Barker-Rhetorical} for details). 
\end{defn}

It is shown in 
\cite[pg. 3816]{Barker-Rhetorical} that this definition of a biset functor as an $R\Gamma ^{\cX}$-module coincides 
with the definition given in Definition \ref{def:biset}.

By \cite[Lemma 2.3.26]{Bouc-Book}, any transitive $(K,H)$-biset is a composition of five elementary bisets.  
To introduce our notation, let $H\le G$  and $N\unlhd G$ be subgroups of $G$ and $\phi: G'\to G$ be an isomorphism. 
Then the set $G$ can be regarded as a $(G,H)$-biset, as an $(H, G)$-biset or as a $(G, G')$-biset via the usual 
actions. These bisets are called the induction biset $\ind_H^G$, the restriction biset $\res^G_H$ and the isogation biset 
$\iso_{G, G'}^\phi$, respectively. In addition, the set $G/N$ can be regarded as a $(G, G/N)$-biset or as a 
$(G/N, G)$-biset. These are called the inflation biset $\infl_{G/N}^G$ and deflation biset $\defl_{G/N}^G$. With this notation, 
for a subgroup $J \le K\times H$, we have 
\[
\Big(\frac{K\times H}{J}\Big) = \ind_P^K \infl_{P/A}^P\iso_{P/A, Q/B}^\lambda\defl^{Q}_{Q/B}\res^H_Q
\]
where $P=\{ k \in K\, |\, \exists h \in H : (k,h) \in J \} $ and  $A=\{ k \in K \, | \, (k,1) \in J\}$.
The subgroups $Q$ and $B$ are defined in a similar way. The isomorphism $\lambda: Q/B \to P/A$ 
is the one induced by $J$. 

In particular, the alchemic algebra for $\cX$ over $R$ is generated by its elementary bisets. An $R$-free basis 
for $R\Gamma^\cX$ can be found in \cite[Lemma 2.2]{Webb}.  


\subsection{Destriction algebra}\label{sect:destriction}
Let $R$ be a commutative ring with unity and $G$ be a finite group.
Let $\cX^s$ denote a collection of finite groups defined in the previous section. 
For a section $(U, V)$ of $G$, the composition 
of the deflation map $\defl ^U _{U/V}$ and the restriction map  $\res^G _U$ is called the destriction biset
and is denoted by $\defres^G _{U/V}$ or shortly by $\des^G _{U/V}$. For a section $(M, L)$ of $G$, if $(U, V)$ is another section of 
$G$ with $L\le V\le U\le M$, then there is a $(U/V, M/L)$-biset given by 
$$\iso_{U/V, (U/L)/(V/L)}^\lambda \circ \defl^{U/L} _{(U/L)/(V/L)} \circ  \res^{M/L}_{U/L}$$
where $\lambda$ is the canonical isomorphism. With an abuse of notation, we denote this composition by 
$\des^{M/L}_{U/V}$, and call it the destriction $(U/V, M/L)$-biset.  

\begin{defn}\label{defn:destriction}  Let $R, G$ and $\cX^s$ be as above.
\begin{enumerate}
\item The \textbf{$R$-destriction category} $R\mathcal {DES}_G$ is the subcategory of the $R$-biset category $R\cB_G$
with objects the groups in $\cX^s$, in which morphisms from $H=M/L$ to $K=U/V$ are given by the bisets of the form
$$\rc_{K,K^g}^g\des^{H}_{K^g}$$ where $(U,V)^g \squ (M,L)$ and $g\in G$.

\item The \textbf{native destriction algebra} $\nabla' _G (R)$ is defined as the subalgebra of the 
alchemic algebra $R\Gamma^{\cX^s}$ generated by the bisets of the form $\rc_{ \leftexp{g} S, S}^g = 
\iso_{\leftexp{g} S, S}^g$ and $\des^T_{S}$ where  $g\in G$ and $S=U/V$ and $T=M/L$  with $(U,V)\squ (M,L)$. 
\end{enumerate}
\end{defn}

Note that it follows from \cite[Lemma 2.2]{Webb} that the set
$$\{ \rc_{K,K^g}^g\des^{H}_{K^g}\, |\, K = U/V, H= M/L \text{ with } (U, V)^g \squ  (M,L), g\in U\backslash G\}$$
is an $R$-free basis for $\nabla_G(R)$. Also as in the case of biset functors and the alchemic algebra, the category 
of modules over $\nabla' _G(R)$ can be identified by the category of $R$-linear functors $R\mathcal{DES}_G \to R\text{-mod}$. There is a restriction functor from the category of biset functors to the category of 
native destriction functors, induced by the inclusion map $\nabla '_G(R) \hookrightarrow R\Gamma ^{\cX^s}$ of the subalgebra. 

Now we define an abstract destriction algebra with generators and relations that  
will be more useful for our purposes.

\begin{defn}\label{pro:DesAlg}
Let $R, G$ and $\cX^s$ be as above.
\begin{enumerate}
\item The \textbf{destriction algebra} $\nabla_G(R)$ is the algebra generated by the symbols of the form 
\begin{enumerate}
  \item[(G1)] $\rc_{\leftexp{g} K,K}^g$ for each $g\in G$ and $K = U/V$ in $\cX^s$, called a \emph{conjugation} and
  \item[(G2)] $\des^H_K$ for each pair $K= U/V$ and $H =M/L$ in $\cX^s$ such that  $(U, V) \squ  (M, L)$, called a \emph{destriction} (in this case we also write $K \squ H$)
\end{enumerate}
subject to the following relations
\begin{enumerate}
  \item[(R1)] $\rc_{H,H}^h = \des^H_H$ for any $H = M/L$ and $h\in M$,
  \item[(R2)] $\rc_{\leftexp{gg^\prime}H, \leftexp{g^\prime}H}^{g}\rc^{g^\prime}_{\leftexp{g^\prime}H,H} = 
  \rc_{\leftexp{gg'}H,H}^{gg^\prime} $ and $\des^K_L\des^H_K = \des^H_L$ for any
  $g,g^\prime \in G$ and $L\squ K\squ H$,
  \item[(R3)] $\des^H_K\rc^{g}_{H,H^g} = \rc_{K,K^g}^g\des^{H^g}_{K^g}$ for any $g\in G$ and $K\squ H$, and
  \item[(R4)] $1 = \sum_{U/V\in\cX^s}\rc_{U/V}$ where $\rc_{U/V} =\rc_{U/V, U/V}^1$ is the conjugation 
  associated to the identity element $1$ of $G$.
\end{enumerate}
  \item A \emph{destriction functor} for $G$ over $R$ is defined as a module over the destriction algebra $\nabla _G(R)$.
\end{enumerate}
\end{defn}
It is straightforward to show, using arguments in \cite[Lemma 2.2]{Webb}, that the destriction algebra has a basis
consisting of symbols of the form $\rc_{K,K^g}^g\des^{H}_{K^g}$. Hence there is a homomorphism of algebras 
$\varphi: \nabla_G (R) \to R\Gamma ^{\cX^s}$ defined by mapping the symbol $\rc_{K,K^g}^g\des^{H}_{K^g}$ to the 
biset $\rc_{K,K^g}^g\des^{H}_{K^g}$. By definition the image of the homomorphism $\varphi$ is the native destriction
algebra $\nabla ' _G (R)$. The homomorphism $\varphi$ is not injective in general since some conjugations may 
induce isomorphic bisets, for example, if $g $ centralizes $K$, then $\rc_{K, K^g } ^g=\rc_{K, K} ^1 $ as bisets. 
Every biset functor $F$ can be regarded as a destriction functor
for $G$ over $R$ via the homomorphism $\varphi: \nabla_G (R) \to R\Gamma ^{\cX^s}$.
 
In the next section, we define a small category and express destriction functors 
as functors from this small category to the category of $R$-modules. 
 
 
\subsection{The orbit category of sections}\label{sec:OrbSub} 

Let $G$ be a finite group. We call a set of sections of $G$ a \emph{collection
of sections} of $G$ if it is closed under conjugation. Let $\mathcal C$ be a collection of sections 
of a finite group $G$. The {\bf orbit category 
of sections $\cD^{\mathcal C} _G$} over the collection $\mathcal C$ is defined 
as in Definition \ref{def:OrbCatSec}. When $\mathcal C$ is the collection of all sections of $G$, we denote the category of sections with $\cD_G$.
Note that the category $\cD_G^{\mathcal C}$ is an EI-category (all endomorphisms are isomorphisms) and 
we have the isomorphism Aut$_{\cD_G}\bigl ( (U, V) \bigr ) \cong N_G(U, V)/U$ of groups where 
$N_G(U,V) = N_G(U)\cap N_G(V)$.

\begin{rem} 
Recall that the orbit category $\mathcal O ^{\mathcal C} _G$ of a finite group $G$, over a collection $\mathcal C$ 
of subgroups of $G$, is the category whose objects are the subgroups of $G$ and whose morphisms from 
$H$ to $K$ are the set of $G$-maps $G/H \rightarrow G/K$, with a composition given by the composition of functions. 
We can identify the set of morphisms from $H$ to $K$ with right cosets $Kg$ with $g\in G$ satisfying $gHg^{-1}\leq K$. 
Note that given a collection of subgroups $\cC$, if we define a collection of sections $\cC'$ as a collection formed by 
sections $S=(U, V)$ with $V=1$ and $U \in \cC$, then the category of sections $\cD^{\cC'}_G$ is isomorphic to the 
orbit category $\cO^{\cC}_G$. The notion of orbit category of sections is a generalization of the orbit 
category of a group.
\end{rem}
 
The category $\cD_G$ is closely connected to the destriction algebra $\nabla_G (R)$  introduced in 
Section \ref{sect:destriction}.  
 
\begin{pro}\label{pro:Equivalence}
Let $\cC$ be a collection of sections closed under taking subsections. Let $\overline{\cD^{\cC}_G}$ denote the 
opposite category of  $\cD^{\cC}_G$, and let $1_\nabla^\cC$ denote the sum $\sum_{(U, V) \in\cC} \rc_{U/V}$ 
in the destriction algebra $\nabla_G(R)$. Set $\nabla_G^\cC(R) = 1_\nabla^\cC \nabla_G(R) 1_\nabla^\cC$. Then 
the map $$j_G :R\overline{\cD^\cC_G}\rightarrow \nabla_G^\cC(R)$$ that takes the morphism $Mg$ from $(U,V)$ 
to $(M,L)$  in $\cD_G^{\cC}$ to the basis element $\rc^{g^{-1} } _{S, {}^gS }\des^T_{{}^gS}$ in $\nabla ^{\cC} _G(R)$, 
where $S=U/V$ and $T=M/L$, is an isomorphism of $R$-algebras.
\end{pro}

\begin{proof}  The map $j_G$  is well-defined since, given $x\in M$, we have
\[
\rc^{(xg)^{-1} } _{S, {}^{xg}S }\des^T_{{}^{xg}S} = \rc^{g^{-1} } _{S, {}^gS }\rc^{x^{-1} } _{{}^gS, {}^{xg}S }\des^T_{{}^{xg}S} 
= \rc^{g^{-1} } _{S, {}^gS }\des^{T^{x}}_{{}^{g}S}\rc^{x^{-1} } _{T^{x}, T} = \rc^{g^{-1} } _{S, {}^gS }\des^{T}_{{}^{g}S}
\]
by the composition product formula for bisets. Also, by the transitivity of destriction and conjugation, $j_G$ 
preserves products. Therefore $j_G$ is an algebra map, which is clearly surjective. To prove that it is also injective, let 
$\rc^{g^{-1} } _{S, {}^{g}S }\des^T_{{}^{g}S} = \rc^{h^{-1} } _{S, {}^hS }\des^T_{{}^{h}S}$
for $Mg, Mh\in \hom_{\overline{\cD} _G}(T,S)$. This equality holds if there exists $x\in M$ such that 
${}^{xg}S = {}^hS$ and $$\rc^{g^{-1} } _{S, {}^{g}S } = \rc^{h^{-1}} _{S, {}^{h}S}\rc^{x} _{{}^{xg}S, {}^{g}S}.$$
In particular we have $gh^{-1}x\in {}^g U\le M$. Since $x\in M$, we get $gh^{-1}\in M$, that is, $Mg = Mh$, 
as required.
\end{proof}


\subsection{Modules over destriction algebra}
\label{subsec:Modules}

Given a biset functor or a destriction functor $F$, there is an associated 
contravariant functor $$F: \cD^{\mathcal C} _G \to R\text{-mod},$$ defined as follows:
We first consider the biset functor $F$ as an $R\overline{\cD^{\cC} _G}$-module via the composition
$$R \overline{\cD^{\mathcal C} _G} \maprt{j_G} \nabla ^{\cC} _G (R) \maprt{\varphi} R \Gamma ^{\cC} $$
where the isomorphism $j_G$ is given in Proposition \ref{pro:Equivalence}. Then we use the correspondence 
between the contravariant functors $\cD^{\mathcal C} _G \to R\text{-mod}$ and the modules over the category 
algebra $R\overline{\cD_G ^{\cC}}$ (see \cite[Thm 7.1]{Mitchell}). Because of this correspondence 
contravariant functors $\cD_G ^{\cC} \to R$-mod are usually 
called (right)  $R\cD^{\cC}_G$-modules. Conversely, every contravariant functor $F : \cD^{\cC} _G \to R\text{-mod}$ 
can be considered as a $\nabla _G ^{\cC} (R)$-module (destriction functor) via the isomorphism $j_G$.

Note that under these identifications, 
the category algebra of the orbit category $\cO ^{\cC} _G$ over a collection $\cC$ of subgroups of $G$ 
is identified with the algebra $$\rho_G: = \nabla ^{\cC'} _G (R)=1_\nabla^{\cC'} \nabla _G (R) 1_\nabla^{\cC'},$$ 
where $\cC'=\{ (H,1) \, | \, H\in \cC \}$. The algebra $\rho_G$ is called the restriction algebra 
for $G$ over $R$. The representation  theory of $\rho_G$ is well-known. Moreover representation theory 
of $\nabla _G ^{\cC} (R)$ for an arbitrary collection $\cC$ is very similar to that of the restriction algebra. 
Simple and projective modules can be described 
using the following constructions. Alternatively, one can use the general results about the representations of 
EI-categories to construct these modules.  

Let $\nabla :=\nabla ^{\cC} _G (R)$. There is a subalgebra $\Omega$ of $\nabla$, called the \emph{conjugation subalgebra}, generated 
by all conjugations in $\nabla$. It is isomorphic to the quotient of $\nabla$ by the ideal generated by all proper destriction 
maps, that is by all $\des^H_K$ with $H\neq K$. This means that there are several functors between the categories 
of modules of these algebras, namely there are induction, coinduction and inflation from $\Omega$-mod to 
$\nabla$-mod and deflation, codeflation and restriction from $\nabla$-mod to $\Omega$-mod. Two of these 
functors are easy to identify. The restriction functor $\res^\nabla_\Omega$ from $\nabla$-modules to 
$\Omega$-modules is the usual restriction of the $\nabla$-action. The inflation functor $\infl_\Omega^\nabla$ 
is the restriction functor along the canonical epimorphism $\nabla\to\Omega$ of algebras.  

First we describe the $\Omega$-modules. Given a section $H=(U,V)$ of $G$, 
we write $\rc_{[H]}$ for the sum of 
all idempotents $\rc_{M/L}$ over all sections $(M,L)$ which are $G$-conjugate to $(U,V)$. Then $\Omega_{[H]} =\rc_{[H]}\Omega\rc_{[H]}$ 
is a two-sided ideal of $\Omega$ and \[ \Omega = \bigoplus_{H\squ_GG} \Omega_{[H]}. \] 
In particular, any $\Omega$-module $F$ can be decomposed as $F = \oplus_{H\squ_G G} F_{[H]}$ where 
$F_{[H]} = \rc_{[H]}F$. 

Given a section $H$ of $G$, the algebra $\Omega_H := \rc_H\Omega\rc_H = \rc_H\Omega_{[H]}\rc_H$ is isomorphic to the group algebra 
$R[N_G(U,V)/U]$ and is also a (non-unital) subalgebra of $\Omega_{[H]}$. Hence there is an obvious restriction functor 
$\res^{\Omega_{[H]}}_{\Omega_{H}}:\Omega_{[H]}$-mod $\to \Omega_H$-mod given on objects by 
$\rc_H\Omega_{[H]}\otimes_{\Omega_{[H]}} -$. Its left adjoint is the induction functor
$\ind^{\Omega_{[H]}}_{\Omega_{H}} = \Omega_{[H]}\rc_H\otimes_{\Omega_H} -$. It is easy to show that there are
natural equivalences $$\ind^{\Omega_{[H]}}_{\Omega_{H}}\circ\res^{\Omega_{[H]}}_{\Omega_{H}} \cong {\rm Id}_{\Omega_{[H]}-{\rm mod}}\,\,{\rm and}\,\, \res^{\Omega_{[H]}}_{\Omega_{H}}\circ\ind^{\Omega_{[H]}}_{\Omega_{H}}\cong{\rm Id}_{\Omega_{H}-{\rm mod}}.$$
In particular the algebra $\Omega_{[H]}$ is Morita equivalent to the group algebra $\Omega_H$. 

Now it is straightforward to see that the algebra $\Omega$ is Morita equivalent to the product algebra 
$\prod_{H\squ_G G}\Omega_H$. Hence the representations of $\Omega$ are just tuples of group representations 
for the various groups appearing in the product. In particular, the simple $\Omega$-modules are the functors 
$S_{H, I}^\Omega$ where $H\squ G$ and $I$ is a simple $\Omega_H$-module such that $\rc_{[K]}S_{H, I}^\Omega =0$
if $K$ is not conjugate to $H$ in $G$ and $S_{H, I}^\Omega(H) = I$. We call such a functor which has a support only on a unique conjugacy class of sections 
an \emph{atomic} functor. 
Similarly, if $P_I$ is a projective 
cover of $I$ as an $\Omega_H$-module, then the atomic functor $P_{H, I}^\Omega$ with value $P_I$ at $H$ is a projective cover of 
$S_{H,I}^\Omega$. We also have
\[
S_{H,I}^\Omega = \ind_{\Omega_{H}}^{\Omega} I \;\;\; \mbox{\rm and} \,\,\, P_{H,I}^\Omega = \ind_{\Omega_{H}}^{\Omega} P_{I}.
\]
With this notation we have the following description of simple and projective indecomposable $\nabla$-modules. Note that although the classification of simple and projective modules for $\nabla$ are similar to that of biset functors as in \cite{Bouc-Book}, the following result does not follow from the later classification. In \cite{Bouc-Book}, the categories of bisets always contains all inductions and restrictions hence $\nabla$ is not an example of these general results.

\begin{pro}
Let $H\squ G$ be a section of $G$, let $I$ be a simple $\Omega_H$-module and $P_I$ be its projective cover. 
The following statements hold.
\begin{enumerate}
\item The $\nabla$-module $S_{H,I}^\nabla = \infl_{\Omega}^\nabla S_{H,I}^\Omega$ is simple. Moreover, any simple 
$\nabla$-module is of this form for some pair $(H,I)$ with $H\squ G$ and $I$ a simple $\Omega_H$-module.
\item The $\nabla$-module $P_{H,I}^\nabla = \ind_\Omega^\nabla P_{H,I}^\Omega$ is projective indecomposable 
with simple head $S_{H,I}^\nabla$. 
\end{enumerate}
\end{pro}

\begin{proof}
It is clear that the $\nabla$-module $S_{H,I}^\nabla$ is simple. To see that any simple $\nabla$-module is of this form, 
note that for any $\nabla$-module $F$ and any sections $(U,V)$ and $(M,L)$ of $G$ with $F(U/V) \neq 0 \neq F(M/L)$ such that 
$|U/V| \ge |M/L|$ and $(U,V) \neq_G (M,L)$, the $R$-module $F(M/L)$ generates a proper non-zero $\nabla$-submodule of $F$. 
Thus a simple $\nabla$-module must be atomic. Moreover if $S$ is a simple $\nabla$-module with $S(H) \neq 0$, 
then clearly $S(H)$ is simple. 

To prove the second part, write
\[
P_{H,I}^\nabla = \ind_\Omega^\nabla P_{H,I}^\Omega = \ind_\Omega^\nabla \ind^\Omega_{\Omega_H} P_I  
= \ind_{\Omega_H}^\nabla P_I = \nabla\otimes_{\Omega_H} P_I.
\]
Since $\nabla$ is a free right $\Omega_H$-module, by general properties of tensor product, the 
$\nabla$-module $P_{H,I}^\nabla$ is projective and it is indecomposable 
since we have the following isomorphisms of rings.
\[
\mbox{\rm End}_\nabla(P_{H,I}^\nabla) = \hom_{\nabla}(\nabla\otimes_{\Omega_H} P_I, \nabla\otimes_{\Omega_H} P_I) 
\cong \hom_{\Omega_H}(P_I, \nabla\otimes_{\Omega_H} P_I) \cong \hom_{\Omega_H}(P_I, P_I).
\]
Here the last isomorphism holds since as $\Omega_H$-modules,
there is an isomorphism
$$\nabla\otimes_{\Omega_H} P_I  = \rc_H\nabla\rc_H\otimes_{\Omega_H} P_I \cong\Omega_H\otimes_{\Omega_H} P_I \cong P_I.$$  
Since by definition, $P_I$ is indecomposable, the ring $\mbox{\rm End}_\nabla(P_{H,I}^\nabla)$ 
is local, and hence $P_{H,I}^\nabla$ is indecomposable. 
\end{proof}

\begin{cor}
Let $H\squ G$ and $I$ be a simple $\Omega_H$-module. Then for any  section $K\squ G$, there is an isomorphism 
\[
P_{H, I}^\nabla (K) \cong R\hom_\nabla(H, K)\otimes_{\Omega_H} P_I
\]
of $\Omega_K$-modules.
\end{cor}
 
This gives a characterization of projective indecomposable $R\cD^{\cC}_G$-modules via the isomorphism between 
$R\overline \cD_G^{\cC}$ and $\nabla^{\cC} _{G}(R)$. Alternatively we can define projective $R\cD_G^{\cC}$-modules 
as direct summands of free $R\cD_G ^{\cC}$-modules, which are defined as follows:

\begin{defn} 
For any section $(U,V)$ of $G$, let $P_{(U,V)}$ denote the $R\cD_G ^{\mathcal C}$-module
$$P_{(U,V)} (-) :=R\hom _{\cD _G } ( - , (U,V) )$$ with destriction and conjugation maps defined by composition. 
\end{defn}

By the Yoneda lemma these modules satisfy the property that if $(U, V) \in \mathcal C$, then for every 
$R\cD_G ^{\mathcal C}$-module $F$, 
$$ \hom _{R\cD_G ^{\mathcal C}} (P_{(U, V)} , F)\cong F (U/V) $$
This gives, in particular, that $P_{(U,V)}$ is a projective $RD_{G} ^{\mathcal C}$-module. 
Note that this definition works more generally for any EI-category. We will be using these definitions 
in the next section when we are writing projective resolutions for $R\cD_G ^{\cC}$-modules. 

\begin{rem} Note that to simplify the notation we write $F(U/V)$ for $F\bigl ( (U,V) \bigr ) $ when $F$ 
is an $R\cD_G$-module. We also write $F(G)$ for $F(G/1)$ to simplify the notation further.
\end{rem}


\section{Obstruction groups for the gluing problem}\label{sect:ObsGluing} 

In this section we introduce an obstruction theory for the gluing problem stated in Definition \ref{def:GluingProblem}. 
Let $G$ be a finite group, and let $\cD_G$ denote the orbit category of sections in $G$. Let $\underline{R}$ denote
the \emph{constant functor}, that is, the $R\cD_G$-module with values isomorphic to the trivial module $R$ 
at every section $(U,V)$, where all the destriction and 
conjugation maps are the identity maps. 

\begin{lem}\label{lem:isom1} The constant functor  $\underline R$ is isomorphic to $ P_{(G,1)}$. 
In particular, $\underline R$ is a projective $R\cD_G$-module, and there is an isomorphism 
$$\hom _{R\cD_G} (\underline R , F)\cong F ( G) $$ for every $R\cD_G$-module $F$.  
\end{lem}

\begin{proof}  For every $(U, V) \in \cD_G$, there is a unique morphism $f_{(U,V)} \in \hom _{\cD_G} ( (U,V), (G,1))$, 
which is the coset $G\cdot 1$, and these morphisms are mapped to each other under morphisms in $\cD_G$, 
meaning that, if $f: (U, V) \to (M, L)$ is a morphism in $\cD_G$, then $$f^* (f _{(M,L)})=f_{(M,L)} \circ f =f_{(U,V)},$$
where $f^*$ denotes the morphism induced by $f$ in $P_{(G,1)}$.  This gives the isomorphism 
$\underline R \cong P_{ (G, 1) }$. The second part follows from the properties of modules $P_{(U, V)}$ that were
explained in Section \ref{subsec:Modules}.
\end{proof}

Let $\cC$ be a collection of sections of $G$. In the rest of the section we assume that $\cC$ is closed 
under taking subsections, i.e., if $(M,L) \in \cC$ and $(U,V)  \squ (M,L)$, then $(U,V) \in \cC$. Let $J^{\cC}$ denote 
the subfunctor of $\underline R$ that has the value $R$ on the sections $(U,V)$ in $\cC$, and is equal to zero on 
the other sections. Let $T^{\cC}$ denote the quotient module $\underline R/J^{\cC}$.  Note that for any $R\cD_G$-module 
$F$, we have $$\hom _{R\cD_G} (J^{\cC} , F) \cong \varprojlim \limits _{(U,V) \in \cD^{\cC}_G} F(U/V).$$ The short exact 
sequence $0\to J^{\cC} \to \underline R \to T^{\cC} \to 0$ gives a 4-term exact sequence that takes the following form 
when we apply the isomorphisms above: 
\[
\xymatrix{0 \ar[r] & \hom_{R\cD_G} (T^{\cC}, F)  \ar[r] & F(G) \ar[r]^-{r^F_G} & \varprojlim \limits _{(U,V) \in \cD^{\cC}_G} 
F(U/V)  \ar[r] & \ext{}^1_{R\cD_G} (T ^{\cC} ,F)  \ar[r] & 0.}
\]
The map $r^F_G$ induced by the inclusion $J^{\cC} \to \underline R$ coincides with the detection map defined by 
$$r^F_G(f)= (\defres ^G _{U/V} f)_{(U,V)\in \cD_G ^{\cC}}. $$
This 4-term exact sequence can be considered as the solution for the gluing problem for a collection $\cC$. 
The obstruction group is given by $$ \obs ^{\cC} (F(G)) \cong \ext{}^1_{R\cD_G} (T ^{\cC},F) $$
and the uniqueness is determined by the group $\hom_{R\cD_G}(T^{\cC} ,F).$ 

\begin{rem} 
Gluing problems for more general collections of sections are considered in  \cite[Section 4]{BoucThev-Glue}. 
One of these collections is the collection $\cC$ of all sections $(U,V)$ where the quotient group $U/V$ is an 
elementary abelian $p$-group. In the case where $G$ is a $p$-group with $p$ odd, it is shown in 
\cite[Corollary 4.3]{BoucThev-Glue} that the detection map $r^F_G$ for this collection is an isomorphism 
for the biset functor $D_t$ of the torsion part of the Dade group. Another collection of sections considered in 
\cite{BoucThev-Glue} is the collection of the sections $(N_G(Q), Q)$ where $Q$ is a $p$-centric subgroup of $G$. 
Note that in this case, the family of sections is not closed under taking subsections. We discuss the gluing problem 
for such collections in the next section.
\end{rem}

We now focus on the case where $\cC$ is the collection of all sections $(U,V)$ of $G$ such that $V\neq 1$. 
In this case we denote the $R\cD_G$-modules $T^{\cC}$ and $J^{\cC}$ simply by $T$ and $J$.  

\begin{pro}\label{pro:isom2} 
Let $F$ be a destriction functor or a biset functor  for $G$ over $R$. Then the map
\begin{eqnarray*}
\phi: \hom_{R\cD_G} (J,F)&\to& \varprojlim_{1<H\le G} F(N_G(H)/H) 
\end{eqnarray*}
defined by $\phi(f)= (f(1_{N_G(H)/H}))_{1< H\le G}$ is an isomorphism of $R$-modules. 
\end{pro}

\begin{proof}  It is clear that the map  $\phi$ is additive. To prove that it is well-defined, we need to show that 
$\phi(f)$ is a gluing data for every $f\in \hom _{R\cD_G} (J, F)$, i.e., we need to verify that $\phi (f)$ satisfies 
the conjugation and destriction invariance conditions given in Definition \ref{def:GluingProblem}. The conjugation 
invariance of $\phi(f)$ is trivial and we leave the justification to the reader. To prove the destriction invariance, 
note that for every $K\nsub H\le G$ we have 
\[
 \des^{N_G(H)/H}_{N_G(H, K)/H} 1_{N_G(H)/H}  =1_{N_G(H, K)/H}=\des^{N_G(K)/K}_{N_G(H, K)/H} 1_{N_G(K)/K}, 
 \]
and hence the destriction invariance follows from the fact that $f$ is an $R\cD_G$-module homomorphism.

Now let
\[
\psi: \varprojlim_{1<H\le G} F(N_G(H)/H)\rightarrow\hom_{R\cD_G} (J,F)  
\]
be the map defined by
$\psi \bigl ( (f_H)_{1< H\le G} \bigr )=   (1_{K/L}\mapsto \res^{N_G(L)/L}_{K/L}f_L).$
This map is additive since the restriction map is additive. Moreover it is clear that $\phi$ and $\psi$ are inverse to
each other. Therefore it remains to check that the map $\psi$ is well-defined. Let $f=(f_H)_{1<H\le G}$ be a gluing 
data. We have to prove that $\psi_f := \psi(f)$ is a morphism of $R\cD_G$-modules. By the conjugation invariance 
of the gluing data $f$, it is clear that $\psi_f$ commutes with conjugations. To prove that it also commutes with 
destrictions, let $U/V\squ M/L$.  Then we have
\[
\psi_f(\des^{M/L}_{U/V} 1_{M/L}) = \psi_f(1_{U/V}) = \res^{N_G(V)/V}_{U/V} f_V.
\]
We also have
\[
\des^{M/L}_{U/V} \psi_f(1_{M/L}) = \des^{M/L}_{U/V}\res^{N_G(L)/L}_{M/L} f_L = \des^{N_G(L)/L}_{U/V} f_L.
\]
Thus, it is sufficient to show that
\begin{equation}\label{eqn:needed}
\res^{N_G(V)/V}_{U/V} f_V =  \des^{N_G(L)/L}_{U/V} f_L.
\end{equation}
This follows from the destriction invariance of the gluing data. Indeed, since $L\nsub V$, the destriction 
invariance gives
\[
\des^{N_G(L)/L}_{N_G(V, L)/V}f_L = \res^{N_G(V)/V}_{N_G(V, L)/V}f_V.
\]
Moreover, since $L\nsub U$, we have $U \le N_G(V)\cap N_G(L)= N_G (V, L)$. Thus the equality 
(\ref{eqn:needed}) holds, as required.  
\end{proof}

As a consequence of the above result we obtain the following.
 
\begin{pro}\label{pro:main1}
Let $F$ be a destriction functor or a biset functor for $G$. Let $\cD_G$ denote the orbit category of sections over 
the collection of all sections in $G$, and let $T$ denote the restriction of the constant functor to the set of all sections
$(U, V)$ with $V=1$. Then, as $R$-modules, we have the following isomorphisms
\[\obs(F(G)) \cong \ext_{R\cD_G} ^1 (T, F)   \text{         and         } \ 
\krn (r^F_G) \cong \hom_{R\cD_G}(T,F).\]
\end{pro}

\begin{proof} Consider the diagram
$$\xymatrixcolsep{1pc}\xymatrix{0 \ar[r] & \hom _{R\cD_G} (T, F) \ar[r] \ar@{-->}[d]& 
\hom_{R\cD_G} (\underline R, F) \ar[r]^{\iota^*} \ar[d]^{\cong} & \hom_{R\cD_G} (J, F) \ar[d]_{\phi}^{\cong}  
\ar[r] & \ext ^1 _{R\cD_G} (T, F) \ar[r] \ar@{-->}[d] & 0 \\
0 \ar[r] & \ker (r^F_G) \ar[r]  & F(G)  \ar[r]^-{r^F_G} & \varprojlim \limits _{1<H\le G} F(N_G(H)/H)  \ar[r]  & 
 \obs (F(G)) \ar[r]  & 0 \\}
$$
where the vertical maps in the middle are the isomorphisms defined in Lemma \ref{lem:isom1} and Proposition 
\ref{pro:isom2}. The map $\iota^*$ is the map induced by the inclusion $\iota: J\to \underline R$.  The square 
in the middle commutes, so it induces two maps, one on each end, which are isomorphisms by 
5-lemma.
\end{proof}
 
Note that Propositions \ref{pro:isom2}  and \ref{pro:main1} together complete the proof of Theorem \ref{thm:Main1}. 
We now give two examples to illustrate how the obstruction groups $\obs(F(G))$ can be calculated using specific projective
resolutions. 

\begin{ex} 
Let $p$ be a prime, $n \geq 1$, and $G=C_{p^n}$ be the cyclic group of order $p^n$. Let $Z$ denote the cyclic subgroup of 
order $p$ in $G$. There is a  projective resolution of  $T$ as an $R\cD_G$-module which is of the 
form $$0 \to P_{(G,Z)} \to P_{(G,1)} \to T\to 0.$$ 
Note that in this case $J$ is isomorphic to the projective module $P_{(G,Z)}$. For a biset functor $F$
we have $\ext_{R\cD_G} ^i (T, F)=0$  for $i \geq 2$, and in low dimensions we have a short exact sequence 
of the form
$$ 0 \to \hom _{R\cD_G} (T, F) \to F(G) \maprt{\varphi} F(G/Z) \to \obs(F(G))\to 0,$$
where $\varphi$ is the deflation map $\defl ^G _{G/Z}$. Since $\defl^G _{G/Z} \infl ^G _{G/Z} =\mathrm{id} _{F(G/Z)}$, 
we obtain that in this case  $\obs (F(G))=0$. 

The kernel group $\ker r^F_G$ is not zero in general.  If $F=B^*$ is the dual group of the Burnside group functor, 
then we can consider elements $f\in B^* (G)$ as functions from the set of subgroups of $G$ to integers, which 
are constant on conjugacy classes of subgroups. Under this identification, the map $\varphi $ takes $f$ to $f'$, 
where  $f'(H/Z) = f(H)$ for every $Z\leq H \leq G$. The kernel of $\varphi$ is isomorphic to $\mathbb Z$ generated 
by $e^G_1$, the function with values $e^G_1(H)=0$ for $H \neq 1$, and 
$e^G_1(1)=1$, which is an idempotent element in $B^* (G)$. \qed
\end{ex}

\begin{ex}\label{ex:D8}  
Let $G=D_8$ be the dihedral group of order $8$. Let $E_1, E_2$ denote two distinct elementary abelian subgroups of rank 2, 
and let $Z$ denote the center of $G$, and let $Q_1, Q_2$ denote two distinct representatives of the conjugacy classes of the non-central subgroups 
of order $2$.  A projective resolution for $T$ as an $R\cD_G$-module can be given as 
$$0 \to P_{(E_1, E_1)} \oplus P_{(E_2, E_2)} \to P_{(E_1, Q_1)} \oplus P_{(E_2, Q_2)} \oplus P_{(G, Z)} \to P _{(G, 1)} 
\to T \to 0$$ and the cochain complex $\hom _{R\cD_G} (P_*, F)$ is of the form 
$$0 \to F(G) \to F(G/Z) \oplus F(E_1/Q_1) \oplus F(E_2/Q_2)  \to  F(E_1/E_1) \oplus F(E_2/E_2) \to 0.$$
From this we obtain that $\ext_{R\cD_G} ^i (T, F) =0$  for $i \geq 3$.
Later we will see that when $G$ is a $p$-group, $\ext_{R\cD_G} ^i (T, F)=0$ for $i \geq \rk (G)+1$ (see Proposition \ref{pro:Oliver}). 
Note that the projective resolution that we give above comes from a projective resolution over the Quillen category (see Example \ref{ex:D8Chp5}). 
\qed
\end{ex}


\section{Reduction to smaller collections}\label{sect:Reduction}

Let $G$ be a finite group and $\cC$ be an arbitrary collection of sections of $G$. For an $R\cD_G$-module $F$, the ext-group 
$\ext^n_{R\cD_G} (T^{\cC} , F)$ is defined to be the $n$-th cohomology group of the cochain complex 
$\hom _{R\cD_G} (P_*, F)$ where  $P_*$ is a projective resolution
\begin{equation}\label{eqn:resolution}
\xymatrix{ \cdots \ar[r] & P _2 \ar[r] & P_1 \ar[r] & P_0 \ar[r] & T^{\cC} \ar[r] &  0.}
\end{equation}
of $T^{\cC}$ as an $R\cD_G$-module. 

The constant functor $\underline R \cong P_{(G, 1)}$ is a projective $R\cD_G$-module, and  $T^{\cC}=\underline R /J^{\cC}$ by definition,
hence we may assume that the first step of a projective resolution for $T^{\cC}$ is the sequence 
$0 \to J ^{\cC} \to \underline R \to T ^{\cC} \to 0$. If we take a projective resolution 
$Q_* \to J ^{\cC}$ for $J^{\cC}$ as an $R\cD_G$-module, then by adding the constant functor $\underline R$, 
we obtain a projective resolution for $T^{\cC}$ of the form
\begin{equation}\label{eqn:resolution2}
\xymatrix{ \cdots \ar[r] & Q_2 \ar[r] & Q_1 \ar[r] & Q_0 \ar[r] &  \underline R \ar[r]  & T^{\cC} \ar[r] &  0.}
\end{equation}
From this we can conclude that for every $R\cD_G$-module $F$,
$$\ext _{R\cD_G} ^n (T ^{\cC}, F) \cong \ext_{R\cD_G } ^{n-1} (J^{\cC} , F)$$
for $n \geq 2$. If $\cC$ is a collection closed under taking subsections, a projective resolution of $J^{\cC}$ 
as an $R\cD_G ^{\cC}$-module is also a projective resolution of $J^{\cC}$ as an $R\cD_G$-module. 
This means that we can replace the ext-group above with the ext-group $\ext ^{n-1}_{R\cD_G^{\cC} } (J^{\cC}, F)$. 

As an $R\cD_G^{\cC}$-module
$J^{\cC}$ is equal to the constant functor $\underline R$ of the category $\cD_G ^{\cC}$ (which is also equal to, with abuse of notation,  the restriction of the constant functor $\underline R$ of the category $\cD_G$ to the subcategory $\cD_G^{\cC}$). Recall that the cohomology of the category
$\cD_G ^{\cC}$ with coefficients in an $R\cD_G^{\cC}$-module $F$  is defined by 
$$H^* (\cD_G ^{\cC} ; F ) := \ext ^* _{R\cD_G ^{\cC} } (\underline R, F)  $$
where $\underline R$ denotes the constant functor for the category $\cD_G^{\cC}$.
From these we conclude the following.

\begin{lem}\label{lem:ExtCoh} 
Let $\cC$ be a collection of sections of $G$ closed under taking subsections. Then for every $R\cD_G$-module 
$F$  we have isomorphisms 
$$\ext _{R\cD_G} ^n (T ^{\cC}, F) \cong \ext_{R\cD_G } ^{n-1} (J^{\cC} , F) \cong H^{n-1}( \cD_G ^{\cC} ; F)$$ 
for every $n \geq 2$.  
\end{lem}


In low dimensions we have the $4$-term exact sequence as before:
\[
\xymatrix{0 \ar[r] & \hom_{R\cD_G} (T^{\cC}, F)  \ar[r] & F(G) \ar[r]^-{r^F_G} & \varprojlim \limits _{(U,V) \in 
\cD^{\cC}_G} F(U/V)  \ar[r] & \ext{}^1_{R\cD_G} (T ^{\cC} ,F)  \ar[r] & 0.}
\] 
Note that the inverse limit term is isomorphic to $\hom _{R\cD_G} (J^{\cC}; F)\cong H^0 (\cD_G ; F)$. To extend this short 
exact sequence to an arbitrary collection $\cC$ of subsections, we define  the reduced cohomology of $\cD_G^{\cC}$ as follows.  

\begin{defn}\label{defn:Reduced} 
Let $\cC$ be an arbitrary collection of sections of $G$. For an $R\cD_G$-module $F$, the reduced cohomology 
$\widetilde H^n (\cD_G ^{\cC}; F)$ of the category $\cD_G ^{\cC}$ is defined as the usual cohomology group 
$H^n(\cD_G ^{\cC}; F)$ for $n\geq 1$ and, at dimensions $n=-1$ and $n=0$, it is defined as the kernel and cokernel of the map $r_G^F$. 
We have an exact sequence of the form
\[
\xymatrix{0 \ar[r] &  \widetilde H^{-1} (\cD_G^{\cC} ; F ) \ar[r] &  F(G) \ar[r]^-{r^F_G} &  \varprojlim \limits _{(U,V) 
\in \cD^{\cC}_G} F(U/ V)  \ar[r] & \widetilde H^0 (\cD_G ^{\cC} ; F) \ar[r] & 0 }
\] 
where $r^F_G$ is the map that takes $f\in F(G/1)$ to the tuple $(\defres ^G _{U/V} f)_{U/V\in \cC} $.
\end{defn}

Alternatively, we can define the reduced cohomology as the cohomology groups of a cochain complex. Given 
a collection of sections $\cC$ of a finite group $G$, let $\cC'$ denote the set of all sections $(M,L)$ such that  $(U, V)\in \cC$ for some $(U,V) \squ (M,L)$.  Note that when $\cC$ is nonempty (which we 
always assume), $\cC'$ includes the section $(G,1)$. Let $J^{\cC'}$ denote the constant functor for the category 
$\cD_G ^{\cC'}$. Then $J^{\cC}$ is a submodule of $J^{\cC'}$ as an $R\cD_G ^{\cC'}$-module, 
in particular, the quotient module $J^{\cC'} /J^{\cC}$
is defined.

If $Q_* \to J^{\cC}$ is a projective resolution of $J^{\cC}$ as an $R\cD_G ^{\cC}$-module, then splicing to it the short
exact sequence $0\to J^{\cC} \to J^{\cC'} \to J^{\cC'}/J^{\cC} \to 0$, we obtain a resolution of  $J^{\cC'}/J^{\cC}$ as an $R\cD_G ^{\cC'}$-module of the form
\begin{equation}\label{eqn:resolutionJC'}
\xymatrix{ \cdots \ar[r] & Q_2 \ar[r] & Q_1 \ar[r] & Q_0 \ar[r] &  J^{\cC'}  \ar[r]  & J^{\cC'}/J^{\cC} \ar[r] &  0.}
\end{equation}
Since $(G,1)$ is in the collection  $\cC'$, the constant functor $J^{\cC'}$ is a projective $R\cD_G ^{\cC'}$-module, hence the above
resolution is a projective resolution of $J^{\cC'}/J^{\cC}$ as an $R\cD_G ^{\cC'}$-module. 
Applying $\hom _{R\cD_G ^{\cC'} } (-, F)$ to this projective resolution, we get a cochain complex 
\[
\xymatrix{0 \ar[r] & \hom _{R\cD_G ^{\cC'} } ( J^{\cC'}, F) \ar[r] & \hom _{R\cD_G ^{\cC'} } ( Q_0, F) \ar[r] &
\hom _{R\cD_G ^{\cC'} } ( Q_1, F) \ar[r] & \cdots} 
\]
where
$$\hom _{R\cD_G ^{\cC'} } ( J^{\cC'}, F) \cong F(G/1).$$
It is now easy to see that the cohomology of this cochain complex is isomorphic to the reduced cohomology of 
the category $\cD_G ^{\cC}$. This shows that the reduced cohomology $\widetilde H^n (\cD_G ^{\cC}; F)$
of the category $\cD_G^{\cC}$ can be described as the cohomology of a cochain complex defined  as above.

\begin{rem} 
If $\cC$ is a collection that is closed under taking subsections, then $\cC'$ is the collection of all sections of $G$.
In this case $J^{\cC'}/ J^{\cC }=T^{\cC}$, hence  the reduced cohomology 
$\widetilde H^n (\cD_G ^{\cC} ; F)$ is isomorphic to the ext-group $\ext^{n+1} _{\cD_G} (T^{\cC} , F)$ 
for $n \geq -1$. For an arbitrary collection the reduced cohomology is no longer isomorphic to the ext-groups of $T^{\cC}$, 
instead it is isomorphic to the ext-group  $$\ext^{n+1} _{\cD_G^{\cC'}} (J^{\cC'}/ J^{\cC} , F).$$ 
\end{rem}

For an arbitrary collection $\cC$ we have the following version of our main result in the previous section.

\begin{pro}\label{pro:ReducedHom} 
Let $\cC$ be an arbitrary collection of sections of $G$, and $F$ be a destriction functor or a biset functor 
for $G$ over $R$. Then the obstruction group  $Obs ^{\cC} (F)$ for the gluing problem over $\cC$ is isomorphic 
to the reduced cohomology group $\widetilde H^0 (\cD_G ^{\cC}; F)$ and the uniqueness is determined by 
$\widetilde H^{-1} (\cD_G ^{\cC} ; F)$. 
\end{pro}

In the rest of the section we show that under certain conditions the collection $\cC$ can be replaced 
by a smaller collection of sections without changing its reduced cohomology.  To do this we use a theorem 
by Jackowski and Slominska \cite{JackSlom} on the cohomology of categories associated to an isotropy presheaf.  
We first introduce the necessary definitions to state Jackowski and Slominska's theorem.

Let $G$ be a finite group and $W$ be a $G$-poset. Let $S(G)$ denote the $G$-poset of subgroups of $G$ 
with the $G$-action given by conjugation, and the partial order is induced by the inclusion of subgroups. 
A $G$-poset map $d: W \to S(G)$ is called an \emph{isotropy presheaf} 
on $W$ if it satisfies $d(w) \leq G_w:=\{ g\in G\, |\, gw=w\}$ for every $w\in W$. Given an isotropy presheaf $d$ 
on $W$, let $W_d$ denote the category whose objects are the elements of $W$ and where the morphisms 
$w_1 \to w_2$ are given by right cosets $d(w_2)g$ for all $g \in G$ such that $gw_1\leq w_2$.
Note that if we take $W=S(G)$ and $d$ is the identity map, the category $W_d$ is the orbit category of $G$, 
so this definition is quite natural, and the category $W_d$ can be thought of a generalization of the orbit category. 

\begin{lem}\label{lem:Isomorphism} 
Let $Sect(G, \cC)$ denote the $G$-poset of all sections in a collection $\cC$, where the $G$-action is given 
by conjugation. Let $d: Sect(G, \cC) \to S(G)$ be the $G$-map defined by $d\bigl ( (U,V) \bigr )=U\in S(G)$. 
Then $d$ defines an isotropy presheaf on $Sect(G, \cC)$ and the category $Sect(G, \cC)_d$ is isomorphic 
to the category of sections $\cD_G ^{\cC}$.
\end{lem}

\begin{proof} For a section $(U, V) \in \mathcal C$, the isotropy subgroup of the $G$-action on $(U,V)$ is 
$N_G(U, V)=N_G(U)\cap N_G(V)$ which includes $U$ as a subgroup. Hence $d: Sect(G, \mathcal C) \to S(G)$ 
defined by $d \bigl ( (U,V) \bigr )=U$ is an isotropy presheaf. It follows that the category 
$Sect(G, \mathcal C )_d$ is isomorphic to the category $\cD_G ^{\mathcal C}$.
\end{proof}

We now recall a result on isotropy presheaves. Note that for a poset map $f: X \to Y$ the comma category 
$y \backslash f$ is defined as the subposet $y \backslash f :=\{x\in X \, |\, y\leq f(x) \}.$

\begin{pro}[Jackowski-Slominska \cite{JackSlom}]\label{pro:JackSlom}
Let $W' \subseteq W$ be a subposet of $W$, and $i: W' \to W$ denote the inclusion map.  Suppose that 
$d'$ and $d$ are isotropy presheaves defined on  $W'$ and $W$ such that $d'= d \circ i$. If for every 
$w\in W$ the topological realization of the poset $w\backslash i=\{ w' \in W' \, | \,  w\leq w'  \}$ is $R$-acyclic, 
then for every $RW_d$-module $F$, there is an isomorphism $$H^* (W_d ; F ) \cong H^* (W'_{d'}; F\circ i_*)$$
where $i_*: W'_{d'} \to W_d$ is the functor induced by $i$. 
\end{pro} 

\begin{proof} This isomorphism is stated in \cite[Section 6]{JackSlom}  as Equation 6.A. The proof follows from 
Theorem 3.4 and Proposition 5.4 in \cite{JackSlom}.
\end{proof}  


As a consequence of Lemma \ref{lem:Isomorphism} and Proposition \ref{pro:JackSlom} we obtain the following.
 
\begin{pro}\label{pro:Reduction}
Let $G$ be a finite group, and let $\cC '$ and $\cC$  be two families of sections of $G$ such that $\cC' \subseteq \cC$. 
Let $i: Sect(G, \cC') \to Sect(G, \cC)$ denote the inclusion of posets map, and $i_*: \cD_G ^{\cC'} \to \cD_G ^{\cC} $ 
be the induced map. Suppose that for every $w\in \cC$ the poset $w\backslash i=\{ w' \in \cC' \, | \,  w\leq w' \}$ 
has a contractible realization. Then, for every $R\cD^{\cC}_G $-module $F$, there is an isomorphism 
$$\widetilde H^* (\cD_G ^{\cC} ; F ) \cong \widetilde H^* (\cD_G ^{\cC'} ; F\circ i_* ).$$ 
In particular, for a biset functor $F$ for $G$, we have $\obs^{\cC} (F(G)) \cong \obs ^{\cC'} (F(G)).$ 
\end{pro} 

\begin{proof}  Let $W=Sect(G, \cC)$ and $W'=Sect(G, \cC')$, and  let $d$ denote the isotropy presheaf defined 
on $W$ by $d (U,V)=U$ and let $d'=d\circ i$.  By Lemma \ref{lem:Isomorphism}  and Proposition \ref{pro:JackSlom}, 
we have an isomorphism $H^* (\cD_G ^{\cC} ; F ) \cong H^* (\cD_G ^{\cC'} ; F\circ i_d)$. The result for the reduced 
cohomology follows from this by applying the 5-lemma on the long exact sequence for the cohomology of a pair of categories.
\end{proof}

Now we give two applications of Proposition \ref{pro:Reduction}. Let $G$ be a finite group and  $p$ be a fixed prime. 
Let $\cD_G ^{p}$ denote the category of sections over the collection of sections $(P, Q)$ in $G$ such that $P$ is 
a $p$-group and $Q\neq 1$, and let $\cD _G ^e$ denote the full subcategory of $\cD^p_G$ whose objects are 
sections of the form $(U, E)$ where $U$ is a $p$-subgroup of $G$ and $E$ is a nontrivial elementary abelian 
$p$-subgroup of $G$ such that $U \leq C_G(E)$.

\begin{pro}\label{pro:CentralRef} 
For every $R\cD^p_G$-module  $F$, there is an isomorphism  
$$\widetilde H^* (\cD_G ^p ; F) \cong \widetilde H^* (\cD^e_G ; F \circ i_*).$$  
\end{pro}

\begin{proof} Let $W^p$ and $W^e$ denote the posets corresponding to the collections  $\cC^p$ and $\cC^e$ for 
the categories $\cD_G^p$ and $\cD_G^e$, and let $i: W^e \to W^p$ denote the inclusion map. The result will follow 
from Proposition \ref{pro:Reduction} once we show that for every $w\in W^p$, the poset $w\backslash i$ is contractible. 
Let $w=(P,Q)$, then  $w\backslash i$ is the poset of sections of the form $(U, E)$ such that 
$1 \neq E \leq Q \leq P \leq U \leq C_G(E)$. Since $P \leq C_G(E)$, the subgroup $E$ is central in $P$ and it lies in $Q$. 
Let $Z$ denote the subgroup $\Omega _1 (Z(P))\cap Q$, where $\Omega _1 (Z(P))$  denotes the subgroup of $P$ 
formed by elements of $Z(P)$ dividing $p$.  Note that $Z\neq 1$  because $Q$ is normal in $P$. In particular, 
$w \backslash i$ is not empty. For every $(U,E)$ in $w\backslash i$, we have  $$E\leq Z\leq Q \leq P \leq U$$ 
hence  $(P,Z)$ lies in $w \backslash i$, and  $(U, E) \succeq (P,Z)$ for every $(U, E)$ in $w\backslash i$. 
By \cite[1.5]{Quillen}, this means that the poset $w\backslash i$ is canonically contractible to the element $(P,Z)$.  
\end{proof}

Let $G$ be a finite group and $F$ be a biset functor for $G$ over $R$. Let $\cD_G^*$ denote the 
orbit category of sections over all sections $(U,V)$ of $G$ with $V\neq 1$. A series of subgroups 
$1 \leq L_0 \nor L_1 \nor \cdots \nor L_n \leq G$ is called a subnormal series in $G$ if $L_{i-1} \nor L_{i}$ 
for every $i \in \{ 1, \dots, n\}$. The normalizer of a subnormal series $1 \leq L_0 \nor L_1 \nor \cdots \nor L_n \leq G$ 
is the intersection of the normalizers $N_G(L_i)$ over all $i$. For finite groups we have the following 
reduction result.

\begin{pro}\label{pro:Subnormal} 
Let $G$ be a finite group, and $\cD _G ^s$ denote the full subcategory of $\cD_G^*$ whose objects are sections 
of the form $(N, L)$ where $N$ is the normalizer of a subnormal series 
$1 < L_0 \nor L_1 \nor \cdots \nor L_n=L$ in $G$. Then for any $R\cD_G$-module $F$,  there is an isomorphism 
$$\widetilde H^* (\cD _G ^* ; F) \cong \widetilde H^* (\cD _G ^{s} ; F \circ i_* ).$$
\end{pro}

\begin{proof}
Let $W^*$ denote the poset of all sections $(U, V)$ of $G$ with $V \neq 1$, and let $W^s$ denote the poset
of sections of the form $(N, L)$ where $N$ is the normalizer of a subnormal series 
$1< L_0 \nor L_1\nor \cdots \nor L_n = L$. Let $i: W^s \to W^*$ denote the inclusion map. By  Proposition 
\ref{pro:Reduction}, we only need to show that for every $w=(U,V) \in W^*$, the subposet $w\backslash i$ 
is contractible. If $(N, L) $ belongs to $ w\backslash i$ then $N$ normalizes a subnormal series 
$1 < L_0 \nor \cdots \nor L_n=L$ and $L \leq V \leq U \leq N \leq N_G(L)$. This gives in particular that 
$L \nor V$. Let $N'=N\cap N_G(V)$. Then $(N', V)$ lies in $W^s$, because $N'$ is the normalizer of the 
subnormal series $1< L_0 \nor \cdots \nor L_{n} \nor L_{n+1} =V$. We have a zig-zag of inclusions 
$(N, L) \succeq (N', V) \squ ( N_G(V), V)$ in $w\backslash i$, so the poset $w\backslash i$ is contractible, 
in two steps, to the section $(N_G(V), V)$ in $w \backslash i$.  
\end{proof}

In the next section we give a refinement of Proposition \ref{pro:CentralRef} for $p$-groups, and use it to calculate 
the obstruction groups for some well-known biset functors defined on $p$-groups. We believe that Proposition 
\ref{pro:Subnormal} is also a very useful reduction for gluing problems involving finite groups, but we will not 
pursue this direction here. 


\section{Higher limits over the Quillen category}\label{sect:Quillen}

Let $G$ be a finite group, and let $\cD_G ^{*}$ denote the category of sections over the collection $\cC^*$ of 
sections $(U, V)$ in $G$ such that $V \neq 1$. If $G$ is a $p$-group, then the collection $\cC ^p$ is equal to 
$\cC^*$, hence in this case Proposition \ref{pro:CentralRef} gives an isomorphism 
$$\widetilde H^* (\cD_G ^* ; F) \cong \widetilde H^* (\cD_G ^e ; F\circ i_*) $$ for every $R\cD_G$-module $F$.  
It turns out that we can refine this isomorphism further to a smaller collection when $G$ is a $p$-group.   
Let $\cD_G^c$ denote the full subcategory of $\cD_G$ over the collection 
$$ \cC^c=\{ (C_G(E) , E) \, | \,  E\neq 1 \text{  is an elementary abelian $p$-subgroup in $G$} \}.$$ 
Note that the $\cD_G^c$ is the full subcategory of $\cD_G^e$ formed by the maximal elements in $\cC^e$. 
We have the following.

\begin{pro}\label{pro:CentralRef2} 
Let $G$ be a finite $p$-group. Then for any $R\cD^*_G$-module $F$, there is an isomorphism 
$$\widetilde H^* (\cD_G ^* ; F) \cong \widetilde H^* (\cD^c_G ; F \circ j_*)$$
where  $j_*: \cD_G^c \to \cD_G^*$ is the map induced by the inclusion.
\end{pro}

\begin{proof}  It is enough to show that for an $R\cD_G^e$-module $F$, there is an isomorphism 
$$\widetilde H^* (\cD_G ^e ; F) \cong \widetilde H^* (\cD^c_G ; F \circ k_*)$$ where $k_*: \cD_G^c \to \cD_G^e$ 
is the map induced by inclusion. Let $W^c$ and $W^e$ denote the posets corresponding 
to the collections for the categories $\cC^c$ and $\cC^e$, and let $k: W^c \to W^e$ denote the inclusion map.  
Let $w=(U,V)$, then  $w\backslash k$ is the poset of sections of the form $( C_G(E), E)$ such that 
$E \leq V  \leq  U  \leq C_G(E)$. Note that $V$ is also an elementary abelian $p$-group and $U\leq C_G(V)$. 
This gives $$E\leq V \leq U \leq C_G(V) \leq C_G(E)$$ and hence $(C_G(V), V)   \squ (C_G(E), E)$ for every 
$(C_G(E), E)$ in $w\backslash k$. This means that the poset $w\backslash j$ is canonically contractible 
to the element $(C_G(V), V)$. Hence the result follows by Proposition \ref{pro:Reduction}. \end{proof}
 
\begin{rem} 
Note that when $G$ is not a $p$-group the above argument fails because in that case the pair 
$(C_G(V), V)$ may not lie in the collection $\cC^p$ since $C_G(V)$ is not necessarily a $p$-group. 
\end{rem} 
 
Let $G$ be a finite group, and $\cC$ be a collection of elementary abelian $p$-subgroups in $G$. 
The \emph{Quillen category} $\cA^{\cC}_p (G)$ is 
defined as the category of elementary abelian subgroups in $\cC$ with morphisms $E_1 \to E_2$ 
given by compositions of conjugations and inclusions. If $\cC$ is the collection of all nontrivial elementary abelian 
$p$-subgroups in $G$, then we denote the Quillen category by $\cA_p (G)$, or simply by $\cA_p$ 
if $G$ is clear from the context. If we also add the trivial subgroup to the collection of subgroups, 
we denote the Quillen category by $\cA'_p$.

For an $R\cD_G$-module $F$, we can define a covariant $R\cA'_p$-module $\overline F$  by taking
$\overline F (E)=F(C_G(E)/E)$ for every $E\in \cA'_p$. For an $\cA'_p$-module $M$, we define the 
reduced cohomology $\widetilde H^i (\cA_p ; M )$
as the cohomology of the cochain complex obtained by applying $\hom _{R\cA_p'} (- , M) $ to the resolution 
$$ \cdots \to P_2 \to P_1 \to P_0 \to {\underline R} \to 0,$$ where $P_*$ is the projective resolution for the 
constant functor $J$ for $\cA_p$ and $\underline R$ is the constant functor for the category $\cA_p'$.  
Note that $\underline R$ is a projective $\cA_p'$-module because it is isomorphic to the projective module 
$P_1$ defined by $P_1 (-) = R \hom _{\cA_p'}  (1, -)$. 

Now we are ready to prove Theorem \ref{thm:Main2} stated in the introduction.

\begin{pro}\label{thm:SubQuillen} 
Let $G$ be a finite $p$-group. Then, the opposite category $\overline{\cD}_G ^c$ is isomorphic to the category 
$\cA_p$ via the functor which takes $E\in \cA_p$ to $(C_G(E), E ) \in \cD_G ^c$.  As a consequence, for any 
$R\cD_G$-module $F$, there is an isomorphism 
$$\widetilde H^* (\cD^*_G ; F) \cong \widetilde H^* (\cA_p ; \overline F ).$$ 
In particular, $\obs (F(G))\cong \widetilde H^0 (\cA_p; \overline F )$.
\end{pro}

\begin{proof}  The morphism set $\hom _{\cA_p} (E_1, E_2)$ can be identified with
the set of cosets $gC_G (E_1) $ such that $\leftexp{g} E_1 \leq E_2.$ We can identify this set with the set of 
morphisms $$\hom _{\cD_G} ( (C_G(E_2), E_2) , (C_G(E_1), E_1))=\{ C_G (E_1) x \, | \, \leftexp{x} (C_G(E_2),E_2) 
\squ (C_G(E_1), E_1) \}$$ via the bijection which takes $g C_G (E_1)$ to $C_G(E_1) g^{-1}$. The second statement 
follows from Proposition 
\ref{pro:CentralRef2}. 
\end{proof}

The higher limits over the Quillen category of a finite group $G$ can be calculated using a cochain complex defined by Oliver  
\cite[Theorem 1]{Oliver}. This cochain complex is defined using Steinberg modules and it vanishes at dimensions 
above the $p$-rank of $G$. Recall that $p$-rank of a finite group $G$, denoted by $\rk _p(G)$, is defined as 
the largest integer $s$ such that $(\Zz/p)^s \leq G$. The rank of a finite group $G$, denoted by $\rk(G)$, is defined as the 
maximum value of $p$-ranks of $G$ over all primes $p$ dividing the order of $G$. For a $p$-group $G$, the $p$-rank and the 
rank are equal.

For an elementary abelian $p$-subgroup $E$ of $G$, let ${\rm Aut}_G(E)$ denote the quotient group $N_G(E)/C_G(E)$.
The Steinberg module $\St_E$ is the $\aut _G(E)$-module defined as the reduced homology of the poset of proper subgroups of $E$.  
As a direct consequence of Theorem \ref{thm:SubQuillen}, Oliver's theorem gives the following proposition. 
A commutative ring $R$ with unity is called \emph{$p$-local} if for every integer $q$ with $(p,q)=1$, 
the map $m_q: R \to R$ defined by multiplication by $q$ is an isomorphism. 

\begin{pro}\label{pro:Oliver} Let $G$ be a finite $p$-group, $R$ be a $p$-local commutative ring, and $F$ be an 
$R \cD_G$-module. For each $k \geq 0$, let $\cE_k$ denote a set of representatives of the conjugacy classes of elementary 
abelian $p$-subgroups in $G$ with  $\rk (E)=k$ (for $k=0$, $\cE_0$ consists of the trivial group).
Then  for any $i\geq -1$, the reduced cohomology group $\widetilde H^i (\cD_G ; F)$  is isomorphic 
to the $i$-th cohomology of the cochain complex $(C^* _{\St} (F), \delta)$, where
$$ C^i _{\St} (F) \cong \prod _{E \in \cE _{i+1} }  \hom _{\aut_G(E)} (\St _E, F(C_G (E)/E)).$$
In particular, $\widetilde H^i (\cD_G ; F)= 0$ for $i\geq \rk (G)$.
\end{pro}

The $\aut_G (E)$-action on $F(C_G(E)/E)$ is given by the map induced by conjugation
homomorphism $c_n : C_G(E)/E \to C_G(E)/E $ defined by $c_n ( x)= nxn^{-1}$ for every $x\in C_G(E)$ and $n\in N_G(E)$.
The boundary maps $\delta$ can be described as follows: For each $E \in \cA_p$, there is a map  
$$R_E : \St_E \to \bigoplus _{ \substack{A \leq E \\ [E: A]=p} } \St_A$$ defined by truncating chains of subgroups or using the boundary 
maps for cohomology of posets (see \cite[page 1390]{Oliver}). Then the map 
$$\delta ^{i-1}: C_{\St} ^{i-1} (F) \to C^i _{\St} (F)$$ is defined as the map that takes $c\in  C_{\St } ^{i-1} (F)$ 
to the element whose coordinate for $E \in \cE_{i+1} $ is equal to the composite 
$$ \xymatrix{ \St_E \ar[r]^-{R_E}  & \bigoplus \limits_{\substack{A \leq E \\ [E:A]=p}} \St_A \ar[r]^-{\oplus c(A) }  & \bigoplus 
\limits _{ \substack{A \leq E \\  [E:A]=p}} F (C_G(A)/A) \ar[r]^-{\oplus \defres}  & F(C_G(E)/E)}.$$

The following example shows how we can calculate the cohomology of $\cD_G$ using the cochain complex 
in the theorem.

\begin{ex}\label{ex:D8Chp5}  Let $G=D_8$ be the dihedral group of order $8$ as in Example \ref{ex:D8}. In this case  
$$\hom _{\aut_G (E)} (\St _E, F(C_G (E)/E))\cong F(C_G(E)/E)$$ for every elementary abelian 
subgroup $E \leq G$. So the cochain complex given in Proposition \ref{pro:Oliver} is of the form  
$$0 \to F(G) \to F(G/Z) \oplus F(E_1/Q_1) \oplus F(E_2/Q_2)  \to  F(E_1/E_1) \oplus F(E_2/E_2) \to 0.$$
This is the same as the cochain complex in Example \ref{ex:D8} obtained from the projective resolution.
The category $\cA'_p$ is also an EI-category so we can write a projective resolution of $\cA_p'$-modules 
that would give this cochain complex when we apply $\hom _{R \cA'_p} (P_*, \overline F)$ to this resolution. 
For each $E$, let $P_E$ be the projective module where $P_E (-)=R\hom _{\cA'_p} (E, -)$.  Let $T$ denote the 
$R\cA'_p$-module which has value $R$ at the trivial subgroup and zero everywhere else. Then there is a projective 
resolution for $T$ of the form
$$0 \to P_{E_1} \oplus P_{E_2} \to P_{Q_1} \oplus P_{Q_2} \oplus P_{Z} \to P _{1} \to T \to 0.$$ \qed
\end{ex}  

When $R$ is not $p$-local there is a spectral sequence from which one can still calculate the cohomology groups 
$\widetilde H^* (\cD_G; F)$ (see \cite[Proposition 6]{Oliver}). Higher limits over $\cA_p (G)$ can also be calculated 
using a theorem of Grodal \cite[Thm 6.1]{Grodal-High}.  As before we assume that $R$ is a $p$-local commutative 
ring, and $\cC$ is a collection of elementary abelian $p$-subgroups closed under taking overgroups.  Associated to 
an $R\cA_p ^{\cC}(G)$-module $M$, there is a $G$-local coefficient system $\cF$ on $|\cC|$ defined by  
$\cF (E_0\leq \cdots \leq E_n) =M(E_n)$. Grodal's theorem states that for every 
$i \geq 0$, there is an isomorphism $$H^i (\cA_p ^{\cC} (G) ; M) \cong H^i _G (|\cC| ; \cF ),$$
where the second cohomology group is the equivariant cohomology of the simplicial complex $|\cC|$ with the
$G$-local coefficient system $\cF$. This isomorphism can be extended to reduced cohomology groups 
by an easy diagram chasing argument. We obtain the following.

\begin{pro}\label{pro:Grodal} Let $G$ be a finite $p$-group, and $R$ be a $p$-local commutative ring. Let $\cC$ 
denote the collection of sections $(U,V)$ of $G$ with  $V\neq 1$. For every $R \cD_G$-module $F$, 
there is an isomorphism $$\widetilde H^* (\cD^{\cC}_G ; F) \cong \widetilde H^* _G ( |\cC|; \cF).$$
\end{pro}

 As in Proposition \ref{pro:Oliver} we obtain that the reduced cohomology groups 
$\widetilde H^i (\cD_G ; F)$ are zero for all $i \geq \rk(G)$. We also obtain the following.

\begin{pro}\label{pro:Constant} 
Let $G$ be a finite $p$-group and $R$ be a $p$-local commutative ring.  Let $\underline R$ denote the constant 
$R \cD_G$-module. Then $\widetilde H^n (\cD^*_G ;  \underline R)=0$ for all $n$.
\end{pro}

\begin{proof} 
Let $\cE$ denote the collection of all nontrivial elementary abelian $p$-subgroups in $G$.  
By Proposition \ref{pro:Grodal}, we have 
$$\widetilde H^* (\cD^* _G; \underline R) \cong \widetilde H^* _G (|\cE |; \underline R) \cong \widetilde H^* (|\cE | /G ; R).$$ 
The topological realization of  $\cE$ is equivariantly contractible when $G$ is a $p$-group, hence the orbit space 
$|\cE|/G$ is contractible. Thus $\widetilde H^* (\cD_G ^* ; \underline R )=0$. 
\end{proof}

   
\section{Gluing rhetorical $p$-biset functors}\label{sect:Rhetorical}
 
Let $R$ be a commutative ring with unity. A biset functor $F$ over $R$ is called a $p$-biset functor if it is defined 
over the collection of all $p$-groups. The main aim of this section is the calculation of obstruction groups 
for a certain type of $p$-biset functors, called rhetorical $p$-biset functors. We give the definition for rhetorical 
biset functors below in Definition \ref{def:rhetorical}. Some well-known $p$-biset functors, such as the rational 
representation ring functor $R_{\Qq}$ (and its dual $R_{\Qq}^*$) and the functor $D_t$ for the torsion part of 
Dade group, defined over $p$-groups when $p$ is an odd prime, are rhetorical $p$-biset functors. 

Let $F$ be a biset functor for a finite group $G$ over $R$. For a normal subgroup 
$N \nor G$, let $e_N^G$ denote the $(G, G)$-biset $\infl ^G _{G/N} \defl ^G _{G/N}$. The biset 
$e_N^G$ is an idempotent of $RB(G,G)$ since the composition $\defl ^G _{G/N} \infl ^G _{G/N}$ 
is the identity $(G/N, G/N)$-biset. This gives a direct sum decomposition 
$$F(G) \cong e^G_N F(G) \oplus (1-e^G_N)F(G).$$ 

From \cite[Sections 6.2 and 6.3]{Bouc-Book}, for any $M, N \nor G$, 
we have $e^G_N e^G_M=e^G_{NM}$. For each normal subgroup $N \nor G$, we define $f^G_N$ as 
the idempotent $$f^G_N=\sum _{N \leq M \nor G} \mu _{\nor G} (N, M) e^G_M$$ where $\mu _{\nor G}$ 
denotes the M\" obius function for the poset of normal subgroups of $G$. The elements $f_N^G$ are 
orthogonal idempotents of $RB(G,G)$, and $\sum _{N \nor G} f_N ^G =\mathrm{id} _G.$ 
This gives a decomposition for $F(G)$ in the form
$$ F(G)\cong \bigoplus \limits _{N \nor G} f_N ^G F(G).$$ 

\begin{defn} The summand $f_1 ^G F(G)$ corresponding to $N=1$ in this decomposition is called the 
\emph{submodule of faithful elements} of $F(G)$, denoted $\partial F (G)$. Note that 
$$\partial F(G)=\bigcap _{1 < N \nor G} \ker \defl ^G _{G/N}.$$
\end{defn}

From this definition, we immediately see that $\ker r^{F} _G \subseteq \partial F(G)$ where  
$r^F_G$ is the detection map in the gluing problem for $F$. We conclude the following.

\begin{lem} 
Let $F$ be a biset functor for a finite group $G$. If $\partial F (G)=0$, then 
$\hom _{R\cD_G} (T, F)=\ker r^F_G=0$.
\end{lem}

There are many examples of biset functors where $\partial F(G)$ is not zero. If $F=B^*$ is the biset functor 
of the dual of the Burnside ring, then by \cite[Theorem 3.1]{Coskun-GlueBS}  for any $p$-group $G$, we have 
$\ker r^{B^*}_G= \hom _{\Zz \cD_G} (T, B^*) \cong \mathbb Z$, so $\partial B^* (G)$ is not zero. 
If $F=R_{\mathbb Q}^*$ is the biset functor of the dual of the rational representation ring, then 
$\hom_{\Zz\cD_G} (T, R_{\mathbb Q} ^* )=0$ if $G$ is a $p$-group of rank at least 2, and it is equal 
to $\Zz$ otherwise (see \cite[Sec. 4]{Coskun-GlueBS} for details).

A special case is where $G=E$ is an elementary abelian $p$-group. In this case we have 
$$\hom _{R\cD_E } (T, F)=\ker r^F_E=\partial F(E),$$ hence the uniqueness of the solution of the gluing 
problem is equivalent to the existence of faithful elements in $F(E)$. Moreover, in this case there exists 
a solution for the gluing problem for any biset functor.

\begin{pro}[Bouc-Thevenaz \cite{BoucThev-Glue}]\label{pro:EltAb}  
If $F$ is a biset functor defined on an elementary abelian $p$-group $E$, then $\obs (F(E))=0$, i.e., there is 
a solution for the gluing problem for $F$. The solution is unique if and only if $\partial F (E)=0$.  
\end{pro}

\begin{proof}  The proof is the same as proof of \cite[Lemma 2.2]{BoucThev-Glue} given for the functor $D_t$.  
\end{proof}

An important class of $p$-biset functors are rhetorical $p$-biset functors. They are defined as follows.

\begin{defn}\label{def:rhetorical}  
Let $R_{\Qq} (H, K)$ denote the representation ring of $\Qq H$-$\Qq K$-bimodules, and 
$$\lin _{\Qq} : R \otimes _{\Zz} B(H,K) \to R \otimes _{\Zz} R_{\Qq} (H, K)$$ denote the linearization map 
which takes an $(H,K)$-biset $X$ to the rational permutation bimodule $\Qq X$.  A $\Qq G$-Morita functor 
\cite{HTW} or a \emph{rhetorical biset functor} \cite{Barker-Rhetorical} is a biset functor $F$ such that 
$F(u)=0$ for every $u \in \ker \lin _{\Qq}$.
\end{defn}

For rhetorical biset functors we have the following theorem.

\begin{pro}[Hambleton-Taylor-Williams \cite{HTW}] \label{pro:Morita} 
Let $G$ be a $p$-group and $F$ be a rhetorical biset functor for $G$ over $R$. 
Assume $G$ has a normal subgroup $K \cong C_p\times C_p$. Let $C_0, C_1, \dots, C_p$ be the distinct 
cyclic subgroups of $K$ and $Z(G)$ denote the center of $G$.

 (i) If $K$ is central, then the following sequence is split exact
$$ \xymatrix{0 \ar[r] & F(G) \ar[r]^-{\alpha} &  F(G/C_0) \oplus F( G/C_1)  \oplus \dots \oplus F(G/C_p)  
\ar[r]^-{\beta} & \oplus _p  F(G/K)) \ar[r] & 0}.$$

(ii) If $K$ is not central, we may assume that $K  \cap  Z(G) = C_0$. Let $G_0$ denote the
centralizer of $K$ in $G$. Then the following sequence is split exact 

$$ \xymatrix{0 \ar[r] & F(G) \ar[r]^-{\alpha} &  F(G/C_0) \oplus F( G_0 /C_1)  \ar[r]^-{\beta} &  F(G_0 /K) \ar[r] & 0}.$$
\end{pro}
 
The maps $\alpha$ and $\beta$ are defined using deflation and restriction maps 
(see \cite[1.A.16, 5.A.1, and 5.A.3]{HTW}). In (i), the map $\alpha$ is given by $\prod _i \defl ^G _{G/C_i}$, and  
$\beta$ is defined as follows: $\beta |_{F(G/C_i) }= \defl ^{G/C_i} _{G/K}$  
for $1\leq i \leq p$, and $\beta |_{F( G/C_0)} =-\phi $, where $\phi$ is 
the composite $$F(G/C_0) \to F(G/K) \to \oplus _p F(G/K),$$ such that the first map is the 
deflation map and the second map is the diagonal. In (ii), the map $\alpha$ is defined as 
$\defl ^G _{G/C_0} \times \defres^G _{G_0/C_1}$, and  
$\beta : F(G/C_0) \oplus F(G_0/C_1) \to F(G_0 /K) $ is the sum of two maps: $F(G_0/C_1) \to F(G_0 /K)$ 
is the deflation map and the map $F(G/C_0) \to F(G_0/K)$ is the negative of the map $\defres ^{G/C_0} _{G_0/K}$. 

Bouc \cite[Theorem 3.1]{Bouc-Rational} shows that when $F$ is a rhetorical biset functor, the group $F(G)$  
can be expressed as a direct sum $\oplus  _{S\in \cG} \partial F(N_G(S)/S)$ where $\cG$ is a family of subgroups, 
called a \emph{genetic basis}. A $p$-biset functor $F$ is called a \emph{rational $p$-biset functor} if
there exists a genetic basis $\cG$ giving a direct sum decomposition for $F(G)$ (see \cite[Definition 10.1.3]{Bouc-Book}). 
Hence every rhetorical $p$-biset functor is rational. The converse is not true in general, but when $p$ is odd, 
the converse also holds for every $p$-biset functor (see \cite[Theorem 1.3]{Barker-Rhetorical}).
 
The most well-known examples of rhetorical biset functors are the rational representation ring $R_{\Qq}$ 
and its dual $R_{\Qq}^*$. Subfunctors and quotient functors of rhetorical $p$-biset functors are also 
rhetorical \cite[Lemma 3.2]{Barker-Rhetorical}.  Using some exact sequences of $p$-biset functors, 
one can show that the functor of Borel Smith functions $C_b$, the functor of unit group of the Burnside 
ring $B^{\times}$, and the torsion part of the Dade group $D_t^{\Omega}$ are also rhetorical $p$-biset 
functors (see \cite{BoucYalcin}, \cite[Section 3]{Barker-Rhetorical}). Therefore results on obstruction 
groups of rhetorical biset functors will apply to these $p$-biset functors. We now prove some general 
observations for obstruction groups of rhetorical biset functors.
 
\begin{pro}\label{pro:centralK} 
If $F$ is a rhetorical $p$-biset functor, then $\obs(F(G))=0=\ker r^F_G$ for every $p$-group $G$ 
with $\rk (Z(G))\geq 2$.
\end{pro}
  
\begin{proof}  Let $G$ be a $p$-group with $\rk(Z(G))\geq 2$ and $F$ a rhetorical $p$-biset functor. 
Consider the cochain complex $C^* _{\St} (F)$ introduced in Proposition \ref{pro:Oliver}. The first three 
terms of the sequence are given as follows: 
\[
\xymatrix{0 \ar[r] &  F(G) \ar[r]^-{\delta^{-1}}& \bigoplus \limits _{C \in \cE_1} F(C_G(C)/C) \ar[r]^-{\delta^0}&    
\bigoplus \limits _{E \in \cE _2}  \hom _{\aut_G (E)} (\St _E, F(C_G (E)/E))\ar[r] & \cdots} 
\]
where $\aut_G(E)=N_G(E)/C_G(E)$. We have 
$\ker r^F_G\cong \widetilde H^{-1} (\cA_p; \overline F)= \ker \delta ^{-1}$ and $$\obs(F(G))=\widetilde 
H^0 (\cA_p; \overline F)=\ker \delta ^0 / \text{im } \delta ^{-1}.$$ This conclusion holds for a $p$-local 
commutative ring $R$ by Proposition \ref{pro:Oliver}. To see that it also holds for an arbitrary commutative 
ring $R$, consider  the spectral sequence $$E^{i ,j} _1 \cong  \bigoplus _{E\in \cE _{i+1} } H^j (\aut_G(E); 
\hom _R (\St_E, F(C_G(E)/E) )) \Rightarrow \widetilde H^{i+j} (\cA_p; \overline F).$$ Note that the sequence 
$C^* _{\St} (F)$ is the horizontal line $E^{*, 0}_1$ on this spectral sequence. The vertical line at $i=-1$ has 
all zero entries above zero, so the calculation of the reduced 0-th cohomology group 
$\widetilde H^0 (A_p; \overline F)$ is not affected by the possibly nonzero terms at other places. In fact, when 
$G$ is a $p$-group we have $N_G(C)/C_G(C)=1$, hence in this case we also have $E^{0, j}=0$ for all $j>0$. 
Therefore even for the calculation of the first cohomology group $\widetilde H^1 (\cA_p; \overline F)$ we can 
use the cochain complex $C^*_{\St} (F)$ when $G$ is a $p$-group and $R$ is an arbitrary commutative ring.  

Let $K$ be a central subgroup isomorphic to $C_p\times C_p$. The short exact sequence given in part (i) of 
Proposition \ref{pro:Morita} is a subcomplex of this cochain complex. The quotient complex is of the form
$$\xymatrix{0\ar[r] &  0 \ar[r] &  \bigoplus \limits _{C \in \cE'_1} F(C_G(C)/C) \ar[r]^-{\bar \delta^0} &    
\bigoplus \limits _{E \in \cE' _2}  \hom _{\aut_G (E)} (\St _E, F(C_G (E)/E))\ar[r] & \cdots},$$
where $\cE_1'=\cE_1  \backslash \{C_0, \dots, C_p\} $ and $\cE_2'=\cE_2\backslash \{ K\}$. Note that 
this quotient complex has the same cohomology groups as the cochain complex $C^* _{\St} (F)$. 
From this one sees easily that $\ker r^F_G=0$ and 
$\obs(F(G))=\widetilde H^0 (\cA_p ; \overline F )=\ker \bar \delta^0$. 

For each $C \in \cE'_1$, the quotient group $CK/C$ is a central subgroup of $C_G(C)/C$ and it is isomorphic 
to $C_p\times C_p$. Hence for every $C \in \cE'_1$, there is a short exact sequence of the form
\begin{equation*}\label{eqn:central} 
\xymatrix{0 \ar[r] & F(C_G(C)/C) \ar[r]^-{\alpha} & 
\bigoplus \limits _{i=0}^p F(C_G(C) /C_i C) \ar[r]^-{\beta} & \oplus _p F(C_G(C)/ CK) \ar[r] & 0.}
\end{equation*}
The $i$-th component of $\alpha$ coincides with the summand of $\bar \delta ^0$ corresponding to the index 
pair $(C,E)$ where  $E=CC_i \in \cE'_2$.  If $C$ and $C'$ are two different conjugacy class representatives 
in $\cE'_1$ such that $C C_i =C'C_j$, then $i=j$, and $C$ and $C'$ are two subgroups of $E=CC_i$ which 
are not conjugate to each other. Thus we must have $\aut_G(E)=N_G(E)/C_G(E)=1$, and under $\bar \delta ^0$ 
the elements in $F(C_G(C)/C)$ and $F(C_G(C')/C')$ are mapped to two different summands in 
$$\hom _{\aut_G(E) } (\St_E, F(C_G(E)/E) ) \cong \St_E^* \otimes F(C_G(E)/E) \cong \oplus _{p} F(C_G (E)/E).$$  
We conclude that $\obs(F(G))=\ker \bar \delta^0=0$.
\end{proof}

Now we consider the case where $G$ is a finite $p$-group and $Z(G)$ is cyclic. In this case $G$ has a normal subgroup isomorphic to 
$C_p\times C_p$ or $G$ is a Roquette group, i.e., $G$ is a cyclic group $C_{p^n}$, for $n\geq 0$, or generalized 
quaternion group $Q_{2^n}$, for $n \geq 3$, or dihedral group $D_{2^n}$, for $n\geq 4$, or a semi-dihedral group 
$SD_{2^n}$, for $n\geq 4$ (see \cite[Section 9.3]{Bouc-Book}). If $G$ is a Roquette group, then it is easy to write 
explicit projective resolutions and calculate $\obs(F(G))$. If $G$ has a normal subgroup $K \cong C_p\times C_p$, 
then $Z=K\cap Z(G)$ is a central subgroup of order $p$. If $A_1, A_2, \dots, A_p$ denote the subgroups 
of $K$ which are not central in $G$, then $H=C_G(K)=C_G(A_i)$ is a subgroup of  $G$ of index $p$ for all $i$.  
Note that if $g \in G \backslash H$, then $g$ permutes the subgroups $\{ A_1, \dots, A_p\}$.

Let $\cA_{\geq 2} (G)$ denote the poset of elementary abelian subgroups of $G$ of rank greater than or equal to 2. 
This poset is non-empty if and only if  $\rk (G) \geq 2$. If $\rk (Z(G) )\geq 2$, then $\cA_{\geq 2} (G)$ is contractible 
to a  central subgroup $K \cong C_p\times C_p$ by the canonical contractions $K \leq KE \geq E$. If $\rk (Z(G))=1$, 
i.e., $Z(G)$ is cyclic, then  there is a connected component $\cB (G)$ in $\cA_{\geq 2} (G)$ 
such that  any connected component of $\cA_{\geq 2} (G)$ different from $\cB (G)$ consists of a single subgroup 
of rank 2 (see \cite[Lemma 10.21]{GLS}).  

Note that if $\rk (G)=2$, then $A_{\geq 2} (G)$ is just a set of points. In that case, 
we choose one of the rank 2 elementary abelian subgroups as the component $\cB(G)$. If one of these subgroups is normal 
in $G$, we always choose the normal subgroup as the component $\cB (G)$. If $\rk(G)\geq 3$, then $G$ cannot be a Roquette 
group, hence it has a normal subgroup $E \nsub G$ of rank $2$. In this case $E$ lies in $\cB (G)$. 
We summarize the properties of the poset $\cA_{\geq 2}(G)$ in the following lemma 
(see \cite[Lemma 3.1]{BoucThev-Glue}). 
 
\begin{lem}\label{lem:pgroups} 
Let $G$ be a $p$-group with $\rk (G) \geq 2$ and assume that $Z(G)$ is cyclic. \\
(a) $G$ has a unique central subgroup $Z$ of order p. Moreover $Z$ is contained in every
maximal elementary abelian subgroup of $G$. \\
(b) There is a connected component $\cB (G)$ in $\cA_{\geq 2} (G)$ such that  any connected component of 
$\cA_{\geq 2} (G)$ different from $\cB (G)$ consists of a single subgroup of rank 2 (an isolated vertex). \\
(c) If $E$ is a subgroup in $\cA_{\geq  2} (G) - \cB (G)$ and if $S$ is a subgroup of $E$ such that
$E = Z \times S$, then $N_G (S)/S$ is cyclic or generalized quaternion.\\
(d) If $E$ is a subgroup in $\cA_{\geq 2} (G) - \cB (G)$, then all the subgroups $S$ of $E$ of order $p$
distinct from $Z$ are conjugate in $G$.   
\end{lem}   
   
The properties of the poset $\cA_{\geq 2} (G)$ can be used to prove the following.

\begin{thm}\label{thm:Rhetorical} 
Let $F$ be a rhetorical $p$-biset functor. Then for every $p$-group $G$ with $\rk (G) \geq 2$, we have 
$\ker r^F_G=0$ and $$\obs (F(G)) \cong \bigoplus \limits _{S \in \cS}\partial F(C_G(S)/S )$$ where the sum 
is taken over the conjugacy class representatives of cyclic subgroups $S$ such that $E=S\times Z$ belongs 
to $\cA_{\geq 2} (G)- \cB(G)$.  In particular, if  $F$ is such that for every cyclic or quaternion group $H$, $\partial F(H)$ is 
isomorphic to a fixed abelian group $A$, then $$\obs (F(G)) \cong \widetilde H^0 ( \cA_{\geq 2} (G)/G ; A).$$   
\end{thm}   
  
 \begin{proof}  If $\rk ( Z(G)) \geq 2$, then $\cA_{\geq 2} (G)$ is contractible by a $G$-equivariant homotopy, 
 thus $\cA _{\geq 2} (G)/G$ is contractible. Hence in this case the result follows from Lemma \ref{pro:centralK}. 
So we assume that $Z(G)$ is cyclic. The statement is easy to verify for Roquette groups, so let us assume 
that $G$ has a normal subgroup $K\cong C_p\times C_p$. Let $Z$ be the central subgroup of $K$ and 
$A_1, \dots , A_p$ denote the non-central subgroups of $K$ of order $p$. Note that $A_1, \dots , A_p$ 
are all conjugate to each other. Let $H=C_G(K)$.

Consider the cochain complex $C_{\St} ^* (F)$. The short exact sequence given in part (ii) of Proposition 
\ref{pro:Morita}, is a subcomplex of this cochain complex. The quotient complex is of the form
 $$\xymatrix{0\ar[r] &  0 \ar[r] &  \bigoplus \limits _{C \in \cE'_1} F(C_G(C)/C) \ar[r]^-{\bar \delta^0} &    
 \bigoplus \limits _{E \in \cE' _2}  \hom _{\aut_G (E)} (\St _E, F(C_G (E)/E))\ar[r] & \cdots}$$
where $\cE_1'=\cE_1  \backslash \{Z, A_1\} $ and $\cE_2'=\cE_2\backslash \{ K\}$. Note that this quotient 
complex has the same cohomology groups as the cochain complex $C^* _{\St} (F)$. From this we obtain 
that $\ker r^F_G=0$ and $\obs(F(G))=\widetilde H^0 (\cA_p ; \overline F )=\ker \bar \delta^0 $. 

Suppose that $C \in \cE_1'$ is such that $E:=CZ$ is not maximal. This means that $E$ belongs to 
the big component $\cB(G)$ of $\cA_{\geq 2} (G)$. 

Case 1: Assume that $C\leq H=C_G(K)$. Then $K$ centralizes $C$, and $KC/C$ is a normal subgroup 
of $C_G(C)/C$ and is isomorphic to  $C_p \times C_p$. If $KC/C$ is central in $C_G(C)/C$, then we 
have a short exact sequence of the form
\[
\xymatrixcolsep{1.2pc}
\xymatrix{0 \ar[r] & F(C_G(C)/C) \ar[r]^-{\alpha} & F(C_G(C)/CZ) \oplus  \bigoplus \limits _{i=1}^p 
F(C_G(C) /CA_i) \ar[r]^-{\beta} & \oplus _p F(C_G(C)/ CK) \ar[r] & 0.}
\]
If $KC/C$ is not central in $C_G(C)/C$, then we have a short exact sequence of the form
\[
\xymatrixcolsep{1.2pc}\xymatrix{0 \ar[r] & F(C_G(C)/C) \ar[r]^-{\alpha} & F(C_G(C)/CZ) \oplus  
F(C_G(C A_1) /CA_1) \ar[r]^-{\beta} & F(C_G(CK)/ CK) \ar[r] & 0.}
\]
Note that both of these sequences are subcomplexes of the quotient complex of $C_{\St} ^* (F)$. 
So we can take further quotients and assume that the sum in the first term is over the set $\cE_1 '' $ 
of conjugacy class representatives of all  cyclic subgroups $C$ of $G$ such that $C \not \leq H$.

Case 2: Assume $C$ is such that $CZ$ is not maximal, but $C\not \leq H$. Since there is a rank 3 subgroup 
$E' \geq E$, we have $\rk (C_G(C)/C)\geq 2$. Let $\overline C_G(C):=C_G(C)/C$. By induction we assume 
that $\ker r^F_{\overline C_G(C)}=0$. Hence the coboundary map $\delta ^{-1}$ for the group $\overline C_G(C)$ 
is an injective map
$$ \varphi _C :  F(\overline C_G(C)) \to \bigoplus _{\overline D \in \mathcal{T}_1} 
F( C_{\overline C_G(C)} (\overline D) / \overline D),$$ where $\mathcal{T}_1$ denotes the set of 
$\overline C_G(C)$-conjugacy class representatives of subgroups of $\overline C_G(C)$ such that 
$\overline D =D/C \cong C_p$. Since $C$ is not contained in $H$, and $H$ is a subgroup 
of index $p$ in $G$, the subgroup $D$ cannot be a cyclic group 
of order $p^2$, hence we have $D \cong C_p\times C_p$. Let $B= H \cap D$, then we have  $D=BC \cong B \times C$. 
Note that $N_G(D) \leq N_G(B)=C_G(B)$. This implies that $C_{\overline C_G(C)} (\overline D)/\overline D \cong C_G(D)/D$. 
Hence the map 
$$\varphi_C: F(C_G(C)/C) \to \bigoplus \limits _{\overline D \in \mathcal{T}_1} F( C_G(D) / D ) $$ is injective. 
Let $\cE_1''$ denote the set of $G$-conjugacy class representatives of cyclic subgroups $C$ of $G$ considered 
in this case. Then if we take the direct sum of the maps $\varphi _C$ over all elements in $\cE_1''$, we get a map 
from $\oplus _{C\in \cE_1''} F(C_G(C)/C)$ to the group 
$$\bigoplus _{C\in \cE_1''} \bigoplus _{\overline D \in \mathcal{T}_1} F(C_G(D)/D) =\bigoplus_ {D\in \cE_2''} 
\hom _{\aut_G(D)} (\St_D , F(C_G(D)/D)) $$
where $\cE_2 ''$ is the set of $G$-conjugacy classes of rank 2 elementary abelian subgroups of $G$ such that 
$D=(D\cap H) C$ with $CZ$ not maximal and $D \not \leq H$. Note that this map appears in the quotient complex 
of $C_{\St}^*(F)$. Hence by an argument similar to the argument above, we can delete summands of the form 
$F(C_G(C)/C)$ where $E=CZ$ belongs to the big component, but $C\not \leq H$.

The remaining cyclic subgroups $C$ satisfy the conditions (c) and (d) of Lemma \ref{lem:pgroups}. In particular, 
$E=CZ$ is a maximal rank 2 subgroup and in $E$ all subgroups of order $p$ distinct from $Z$ are conjugate. 
For each such $C$, we have a map $$ F(C_G(C)/C) \to F(C_G(C) / CZ).$$ By condition (c), the quotient group 
$L=C_G(C)/C$ is cyclic or isomorphic to a generalized quaternion group, and the kernel of this map is $\partial F (L)$. 
Hence we obtain the isomorphism for $\obs(F(G))$ given in the statement of the theorem. If $\partial F (L)$ is equal 
to a fixed abelian group $A$, then we obtain   $$\obs (F(G)) \cong \widetilde H^0 (\cA_{\geq 2} (G) /G ; A).$$ 
\end{proof} 
 
In the next section we discuss how this result applies to some $p$-biset functors related to endo-permutation 
modules and to the dual of the rational representation ring functor.


\section{Gluing endo-permutation modules}\label{sect:GluingEndo}  
 
Let $G$ be a $p$-group and $k$ a field of characteristic $p$. A $kG$-module $M$ is called an 
\emph{endo-permutation module} if $\hom _k (M, M) \cong M\otimes M^*$ is a permutation $kG$-module. 
An endo-permutation module is \emph{capped} if $M\otimes M^*$ has the trivial module $k$ as a summand. 
Under a suitably defined equivalence relation on endo-permutation modules, the set of equivalence classes 
$[M]$ of capped endo-permutation modules forms an abelian group called the \emph{Dade group} of $G$, denoted  
$D(G)$. The addition in $D(G)$ is defined  by $[M_1]+[M_2]=[M_1\otimes M_2]$ (see \cite[Definition 12.2.8]{Bouc-Book}, 
\cite{Dade}). 

Given a nonempty $G$-set $X$, the kernel of the augmentation map $kX\to k$ is a capped endo-permutation 
module if $|X^G|\neq 1$. In this case, the equivalence class of this endo-permutation module is called a relative 
syzygy and is denoted $\Omega _X$. The subgroup of $D(G)$  generated by relative syzygies $\Omega_{X}$ 
over all finite nonempty $G$-sets $X$ with $X^G=\emptyset$ is denoted $D^{\Omega}(G)$ 
(see \cite[Definition 12.6.3]{Bouc-Book}). The assignment $G \to D(G)$  together with suitably defined restriction, 
induction, isogation, deflation, and inflation maps is not a biset functor because of some Galois twists that occur, 
however the Dade group of relative syzygies $D^{\Omega}(G)$ is a $p$-biset functor (see \cite[Theorem 12.6.8]{Bouc-Book}).  
It is also known that when $p$ is odd we have $D^{\Omega}=D$, and hence in this case $D$ is also a $p$-biset functor 
(see \cite[Theorem 12.9.10]{Bouc-Book}).

The torsion part of the Dade group $D^{\Omega}$ is a subfunctor of $D^{\Omega}$, denoted $D^{\Omega}_t$.
Let $B^* $ denote the dual of the Burnside ring functor, $R_{\Qq} ^*$  the dual 
of the rational representation ring functor, and $C_b$ the functor for the group of Borel-Smith functions. 
Then by Theorem 1.2 and 1.3 in \cite{BoucYalcin}, there is a commutative diagram of $p$-biset functors of the form
\[\xymatrix{ 0 \ar[r] & C_b \ar[r] \ar@{=}[d] & R^* _{\Qq} \ar[r] \ar[d] & D_t ^{\Omega}  \ar[r] \ar[d] & 0 \\
0 \ar[r] & C_b \ar[r] & B^* \ar[r] \ar[d] & D^{\Omega} \ar[r] \ar[d] & 0 \\
& & D^{\Omega} /D^{\Omega }_t \ar@{=}[r]  & D^{\Omega} /D^{\Omega }_t & }\]
 Note that if $p$ is odd, we have $D^{\Omega}=D$.

From this diagram it is clear that the functors $C_b$ and $D_t^{\Omega}$ are both rhetorical $p$-biset functors since 
they are subfunctors and quotient functors of the dual of the rational representation ring functor $R_{\Qq}^*$, which is 
rhetorical (see \cite[Lemma 3.2]{Barker-Rhetorical}). Hence if $F$ is equal to $D_t^{\Omega}$, $C_b$, or $R_{\Qq}^*$, then $\obs(F(G))=0$ for every $p$-group 
$G$ with $\rk (Z(G))\geq 2$. Co\c skun uses this calculation to give another proof for the rank calculation for the group 
of endo-trivial modules $T(G)$ (see \cite[Theorem 6.1]{Coskun-GlueBS}).

Now let us assume that $p$ is odd. Then the torsion part of the Dade group $D_t$ is equal to  $D^{\Omega }_t$, 
hence it is a rhetorical $p$-biset functor. It is also known that $\partial D_t (G) \cong \Zz/2$ for every nontrivial 
cyclic group $G$. So as a corollary of Theorem \ref{thm:Rhetorical}, we recover the following theorem of 
Bouc and Th\' evenaz \cite[Theorem 1.1]{BoucThev-Glue}. 

\begin{pro}[Bouc-Th\' evenaz]\label{pro:BoucThev}  
If $p$ is an odd prime number and $G$ is a $p$-group 
that is not cyclic, then there is a short exact sequence of abelian groups 
\[
\xymatrix{0 \ar[r] & D_t (G) \ar[r]^-{r^{D_t} _G} & \varprojlim \limits_{1<H\le
G} D_t (N_G(H)/H)  \ar[r] &  \widetilde H^0 (\cA_{\geq 2} (G) ; \Ff _2 ) ^G \ar[r] & 0}
\] where $H^0 (\cA_{\geq 2} (G) ; \Ff _2 ) ^G$ denotes the group consists of all $G$-invariant cohomology 
classes under the $G$-action induced by the conjugation action on $\cA_{\geq 2} (G)$.
\end{pro}

\begin{proof}  
By Theorem \ref{thm:Rhetorical}, the map $r^{D_t}_G$ is injective and the obstruction group is isomorphic 
to the reduced cohomology group $\widetilde H^0 (\cA_{\geq 2} (G) / G; \Ff_2 )$.   
Note that  $$H^0 (\cA_{\geq 2} (G) / G; \Ff_2 ) 
\cong H^0 (\cA_{\geq 2} (G); \Ff_2 ) ^G$$ by the definition of the 0-th cohomology. When $p$ 
is odd, there is always a normal subgroup $K \cong C_p \times C_p$, therefore there is a $G$-invariant component
in $\cA_{\geq 2} (G)$. Hence the isomorphism above induces an isomorphism of reduced cohomology classes 
$$\widetilde H^0 (\cA_{\geq 2} (G)/G; \Ff _2 ) \cong \widetilde H^0 (\cA_{\geq 2} (G); \Ff _2 ) ^G.$$ 
\end{proof}

For the dual of the rational representation ring $R_{\Qq} ^*$ we have a similar calculation. For this calculation we 
do not assume that $p$ is odd. If $H$ is a nontrivial cyclic group or a quaternion group, then 
$\partial R_{\Qq} ^* (H) \cong \Zz$. Hence as a direct consequence  of Theorem \ref{thm:Rhetorical}, 
we obtain the following.

\begin{pro}[Co\c skun \cite{Coskun-GlueBS}]\label{pro:ObsR_Q^*} 
If $G$ is $p$-group with $\rk(G) \geq 2$, then
$$\obs (R_{\Qq}^*(G))\cong \widetilde H^0 (\cA_{\geq 2}(G)/G ; \Zz).$$ 
\end{pro}

This is slightly different than the isomorphism given in \cite[Theorem 4.4]{Coskun-GlueBS}. There the obstruction 
group $\obs(R_{\Qq} ^* )$ is given as isomorphic to $\widetilde H^0 (\cA_{\geq 2} (G) ; p\Zz)^G$ which is isomorphic 
to $\widetilde H^0 (\cA_{\geq 2} (G) /G; \Zz)$ in all cases except when $G$ is a rank 2 Roquette group, i.e., when 
$G$ is dihedral or semi-dihedral group. In this case we believe the isomorphism in Proposition \ref{pro:ObsR_Q^*} 
gives the right answer. Note that since the rational representation ring functor $R_{\Qq}$ also has the property 
$\partial R_{\Qq} (H)\cong \Zz$ for all $H$ that is cyclic or isomorphic to a generalized quaternion group, we also have
$$\obs (R_{\Qq}(G))\cong \widetilde H^0 (\cA_{\geq 2}(G)/G ; \Zz)$$
for every $p$-group $G$ with $\rk(G) \geq 2$. 

Using Theorem \ref{thm:Rhetorical} we can also calculate the obstruction group $\obs(D^{\Omega} _t (G))$ when 
$G$ is a $2$-group. 

\begin{pro}\label{pro:ObsDt} 
Let $G$ be a $2$-group with $\rk (G)\geq 2$. Let $\cS$ denote the set of conjugacy class representatives 
of cyclic subgroups $S$ such that $S\times Z$ belongs to $\cA_{\geq 2} (G) -\cB(G)$. 
Then $\ker r^{D_t} _G=0$ and $$\obs (D_t ^{\Omega} (G)) \cong \bigoplus \limits _{S \in \cS} \Zz/ n_S \Zz$$ 
where $$n_S=\begin{cases}  1 & \text{  if  } \ C_G(S)/S\cong C_2 \\ 2 & \text{  if  } \ C_G(S)/S\cong C_{2^n}, 
\ n\geq 2 \\ 4 & \text{  if  } \ C_G(S)/S\cong Q_{2^n}, \ n\geq 3.\\
\end{cases} $$
\end{pro}

\begin{proof} 
This follows from Theorem \ref{thm:Rhetorical} and from the computations of the Dade groups of 
abelian groups and generalized quaternion groups (see \cite[Corollary 12.10.3]{Bouc-Book} and 
\cite[Proposition 12.2]{Dade}). 
\end{proof}

Note that the calculation given above is consistent with the earlier calculations done by Co\c skun 
\cite[Theorem 5.1]{Coskun-GlueBS} for the Borel-Smith functor $C_b$. To see the connection between 
these two calculations, consider the long exact cohomology sequence for the coefficient sequence 
$0 \to C_b \to R_{\Qq}^* \to D^{\Omega}_t\to 0$ given as follows:
$$\xymatrixcolsep{.8pc}
\xymatrix{\cdots \ar[r] & \widetilde H^{-1} (\cD^*_G ; D^{\Omega}_t ) \ar[r] &  
\obs( C_b(G))  \ar[r] & \obs (R_{\Qq} ^* (G)) \ar[r] & \obs( D^{\Omega } _t (G)) \ar[r] &  
H^1 (\cD_G ^* ; C_b ) \ar[r] & \cdots}$$
By Proposition \ref{pro:ObsDt} we have  $\widetilde H^{-1} (\cD_G ^*; D^{\Omega}_t)=\ker r^{D_t^{\Omega}} _G=0$, 
and $$\obs (D_t ^{\Omega} (G)) \cong \bigoplus \limits _{S \in \cS} \Zz/ n_S \Zz$$ 
where $n_S$ is defined as above. By Proposition \ref{pro:ObsR_Q^*}, we have 
$$\obs (R_{\Qq} ^* (G))\cong \widetilde H^0 (\cA_{\geq 2} (G)/G; \Zz) \cong \bigoplus \limits _{S \in \cS} \Zz.$$ 
Hence we recover the obstruction group calculation 
$\obs(C_b(G))=\widetilde H_b ^0 (\cA_{\geq 2} (G), p\Zz) ^G$ given in \cite[Theorem 5.1]{Coskun-GlueBS}.

\begin{rem}\label{rem:Surjective} 
The above discussion on the obstruction group for $C_b(G)$ gives, in particular, that the map  
$\obs (R_{\Qq} ^* (G)) \to \obs( D^{\Omega } _t (G))$, induced by the map 
$R_{\Qq}^* \to D_t ^{\Omega}$, is surjective.
\end{rem}

The rest of the section is devoted to the calculation of the obstruction group for the Dade group functor 
$D$ when $p$ is an odd prime.  In this case $D=D^{\Omega}$, so we may use the short exact sequence 
$0 \to C_b \to B^* \to D^{\Omega} \to 0$ to calculate the obstruction group. The associated long exact 
sequence gives a sequence of the following form
\begin{equation}\label{eqn:ExactSeq}
\cdots  \to \obs( B^*(G) ) \to \obs( D(G) )  \maprt{\delta} H^1 (\cD^*_G ;  C_b) \to H^1 (\cD_G ^* ; B^* ) \to \cdots .
\end{equation}
We claim that the connecting homomorphism $\delta$ is an isomorphism. To prove this we first need a lemma. 

Let $\nabla$ denote the destriction algebra and let $\rc_1^+$ be the sum of the idempotents $\rc_{H/H}$ in $\nabla$ 
as $H$ runs over all subgroups of $G$. Let $\Pi= \rc_1^+\nabla\rc_1^+$ be the truncation of $\nabla$ by $\rc_1^+$. 
The algebra $\Pi$ can be identified with the subalgebra of $\nabla$ generated by conjugations $\rc_{{}^g(H/H), H/H}^g$ 
for all $g\in G$ and $H\le G$. Moreover, it is easy to see that the algebra $\Pi$ is Morita equivalent to the 
product $\prod_{H\le_G G}\Zz [N_G(H)/H]$ of group algebras. 

\begin{lem}\label{lem:Vanish}
Let $B^*$ denote the $p$-biset functor of super class functions and let $Z$ be the $\Pi$-module 
$\prod_{H\le G}\Zz$ with the trivial $\Pi$-action. Then
\begin{enumerate}
\item $B^*\cong \coind_\Pi^\nabla Z$ as $\nabla$-modules. 
\item $\ext^1_\nabla(J,B^*) = 0.$
\end{enumerate}
\end{lem}

\begin{proof} 
By definition, 
\[
\coind_\Pi^\nabla Z = \hom_\Pi(\nabla, Z) = \hom_\Pi(\rc_1^+\nabla, Z)
\]
where we regard $\nabla$ as a $(\Pi, \nabla)$-bimodule via left and right multiplication. Thus given a subgroup
$K\le G$, we have
\[
\bigl (\coind_\Pi^\nabla Z \bigr ) (K) = \hom_\Pi(\rc_1^+\nabla\rc_K, Z).
\]
Now clearly, for any subgroup $H\le K$, there is a unique element in $\rc_1^+\nabla\rc_K$ of the form
$\des^K_{H/H}$. Hence any $f\in \hom_\Pi(\rc_1^+\nabla\rc_K, Z)$ is uniquely determined by its evaluations
$f(\des^K_{H/H})$ as $H$ runs over conjugacy classes of subgroups of $K$. In particular, we obtain a group
homomorphism
\[
\coind_\Pi^\nabla Z (K) \to B^*(K),\ \  f\mapsto (f(H/H))_{H\le K}
\]
for any subgroup $K$ of $G$. It is straightforward to check that this is a functorial isomorphism. Hence the
first part holds. 

To prove the second part, note that since there is the inclusion of rings $\Pi \subseteq \nabla$, we have the
base change spectral sequence for ext-groups (see \cite[Exercise 5.6.3]{Weibel})
\[
E_2^{p,q} = \ext_{\nabla}^p(J, \ext_\Pi^q(\nabla, Z))\Rightarrow \ext_\Pi^{p+q}(J, Z).
\]
The exact sequence in low degrees becomes
\[\xymatrix{ 0 \ar[r] &\ext^1_\nabla(J, \hom_\Pi(\nabla, Z)) \ar[r] & \ext^1_\Pi(J,Z) \ar[r] & \hom_\nabla(J,\ext_\Pi^1(\nabla,Z)) 
\ar[r] & }\]
\[\xymatrix{\ar[r] & \ext^2_\nabla(J, \hom_\Pi(\nabla, Z)) \ar[r] & \ext_\Pi^2(J,Z). }\]
By the first part, $B^* = \hom_\Pi(\nabla, Z)$, hence
it is sufficient to show that $\ext^1_\Pi(J,Z)$ vanishes. But we have
\[
\ext^1_\Pi(J,Z) = \ext^1_\Pi(\rc_1^+J,Z) = \bigoplus_{1 < H\le_G G}\ext^1_\Pi(\rc_{H/H}^+ J, c_{H/H}^+ Z)
\]
where $\rc_{H/H}^+ = \sum_{K=_G H} \rc_{K/K}$. Finally, using the above Morita equivalence and the Morita 
invariance of the ext-groups, we get for any subgroup $H$ of $G$ that $$\ext^1_\Pi(\rc_{H/H}^+ J, 
\rc_{H/H}^+ Z) \cong \ext^1_{RN_G(H)/H}(\Zz,\Zz)=H^1 (N_G(H)/H; \Zz) $$ which is zero since $N_G(H)/H$ 
is a finite group. Thus $\ext_\nabla^1(J,B^*) = 0$, as required.
\end{proof}

Using this lemma we prove the following.

\begin{pro}\label{pro:ObsCb}
Let $D$ denote the functor for the Dade group, and $C_b$ denote the functor for the group of Borel-Smith 
functions. Then
$$\obs( D(G) ) \cong H^1 (\cD_G^*; C_b).$$
\end{pro}

\begin{proof} By \cite[Theorem 3.1]{Coskun-GlueBS}, we have $\obs(B^* (G)) = 0$.  
By the long exact sequence given in Equation (\ref{eqn:ExactSeq}), it is enough to show that $H^1 (\cD_G ^*; B^*)=0$. 
By Lemma \ref{lem:ExtCoh}, we have 
\[
H^1 (\cD_G^*; B^*)\cong \ext^1_{R\cD_G} (J,B^*)
\]
hence the result follows from the Lemma \ref{lem:Vanish}.
\end{proof}

In Bouc \cite[Theorem 2.15]{Bouc-Glue}, it is shown that the obstruction group $\obs(D(G))$ for a $p$-group $G$, 
when $p$ is odd, embeds into the group 
$H^1 (\cA_{\geq 2} (G) ; \Zz) ^{(G)}$ which is defined as the quotient group $Z^1 / B^1$ where 
$Z^1$ is the group of $G$-invariant 1-cocycles in $\cA _{\geq 2} (G)$ and $B^1$ is the subgroup  
generated by the boundaries of $G$-invariant 0-chains in $\cA_{\geq 2} (G)$. We note the following
observation about this group.

\begin{lem}\label{lem:CohIsom} The group $H^1 (\cA_{\geq 2} (G) ; \Zz) ^{(G)}$ defined in \cite[Notation 2.9]{Bouc-Glue}
is isomorphic to the cohomology group $H^1 ( \cA_{\geq 2} (G) /G ; \Zz )$ where $\cA_{\geq 2} (G)/G$ denotes the orbit space
under the conjugation action.
\end{lem}

\begin{proof} Let $X=\cA_{\geq 2} (G)$. Note that the group 
$H^1 (\cA _{\geq 2} (G); \Zz)^{(G)}$ defined above is the cohomology group at dimension 1 of the cochain complex
$$0 \to \hom _{\Zz G} ( C_0 (X), \Zz) \maprt{\delta ^0} \hom _{\Zz G} ( C_1 (X), \Zz)  \maprt{\delta^1} \hom _{\Zz G} ( C_2 (X), \Zz)  \to \cdots.$$ 
Since $$ \hom _{\Zz G} (C_* (X) , \Zz) \cong \hom _{\Zz } (C_* (X) _G , \Zz) \cong \hom _{\Zz} (C_* (X/G) , \Zz),$$
the group $H^1 (\cA _{\geq 2} (G); \Zz)^{(G)}$  is isomorphic to the cohomology group $H^1 (X /G ; \Zz )$.
\end{proof}

In general $\obs (D(G))$ is not 
isomorphic to $H^1(\cA_{\geq 2}(G) /G ; \Zz)$ (see \cite[Section 6]{Bouc-Glue}). In Theorem \ref{thm:Main3} 
we give an explicit description of the obstruction group $\obs (D(G))$ as a subgroup of 
$H^1(\cA_{\geq 2}(G) /G ; \Zz)$. Before giving the proof of Theorem \ref{thm:Main3}, we prove a useful 
technical result.

\begin{pro}\label{pro:H1Calc} Let $F$ be a rhetorical $p$-biset functor such that $\partial F(H)$ is equal to 
a fixed abelian group $A$ for every group $H$ isomorphic to a nontrivial cyclic group or a generalized 
quaternion group. Then there is an isomorphism  $$H^1 (\cD_G^*; F)\cong H^1 (\cA_{\geq 2} (G)/G ; A).$$ 
\end{pro}

\begin{proof} First, let us assume that $\rk (Z(G))\geq 2$. Let $K\cong C_p\times C_p$ be a central subgroup 
in $G$. Then the argument in the proof of Proposition \ref{pro:centralK} gives us that the complex 
$C^* _{\St} (F)$ can be reduced to a complex of the form
\[
\xymatrixcolsep{1.1pc}\xymatrix{0 \ar[r]^-{\bar \delta^0} &    \bigoplus \limits _{E \in \cE'' _2}  \hom _{\aut_G (E)} 
(\St _E, F(C_G (E)/E))\ar[r]^-{\bar \delta ^1}& \bigoplus \limits _{E \in \cE'' _3}  \hom _{\aut_G (E)} (\St _E, F(C_G (E)/E)) 
\ar[r] & \cdots } 
\]
where $\cE_2''$ is the $G$-conjugacy classes of rank 2 elementary abelian subgroups $E$ such that $E\cap K=1$, 
and $\cE_3''$ is the $G$-conjugacy classes of rank 3 elementary abelian subgroups $E$ such that $K \not \leq E$. 
For each $E\in \cE_2 ''$, we have a short exact sequence of the form
\begin{equation}\label{eqn:E} \xymatrix{0 \ar[r] & F(C_G(E)/E) \ar[r]^-{\alpha} &  \bigoplus \limits _{i=0}^p F(C_G(E) /C_i E) 
\ar[r]^-{\beta} & \oplus _p F(C_G(E)/ EK) \ar[r] & 0.} 
\end{equation}
Note that for each $E$, the rank 3 elementary abelian subgroup $C_iE$ does not include $K$.  
This gives that the map $\bar \delta ^1$ is injective, hence $H^1(\cD_G ^*; F)=0$. Since $K$ is central, 
$\cA_{\geq 2} (G)$ is canonically contractible by equivariant homotopies $K \leq KE \geq E$. Therefore 
$H^1 (\cA_{\geq 2} (G) /G ; A)=0$, hence the isomorphism holds in this case.

Now we assume that $Z(G)$ is cyclic. We may also assume that $G$ has a normal subgroup 
$K \cong C_p\times C_p$ because otherwise $G$ is a Roquette group and by direct computation 
we can see that for Roquette groups both cohomology groups are trivial. As before let $Z$ denote 
the central subgroup of $K$ and $A_1, \dots , A_p$ denote the non-central subgroups of $K$ of order $p$.
By the argument we used in the proof of Proposition \ref{thm:Rhetorical} we can reduce the complex 
$C^* _{\St} (F)$ to a complex of the form
\[
\xymatrix{0 \ar[r]^-{\bar\delta^{-1}}&  \bigoplus \limits _{C \in \cE_1''} F(C_G(C)/C)  \ar[r]^-{\bar\delta^{0}} &    
\bigoplus \limits _{E \in \cE'' _2}  \hom _{\aut_G (E)} (\St _E, F(C_G (E)/E)) &  \\
& \ar[r]^-{\bar \delta ^1} & \bigoplus \limits _{E \in \cE'' _3}  \hom _{\aut_G (E)} (\St _E, F(C_G (E)/E)) \ar[r] & \cdots}  
\]
by removing the terms appearing in the short exact sequences of the form given in Case 1 considered in the proof of 
Theorem \ref{thm:Rhetorical}. Note that if a short exact sequence of modules appears as a subcomplex of a chain complex, 
removing the modules of this short exact sequence from the chain complex does not change the homology group of the chain complex.
Assume also that the short exact sequences of the form given in 
(\ref{eqn:E}) earlier are removed for every $E \in \cE'_2$ and $E\in \cE'_3$ satisfying 
$E \leq H=C_G(K)$. 

Finally assume that maps of the form
$$ F(C_G(C)/C) \to F(C_G(C) / CZ)$$
are also removed for all $C$ such that $CZ$ is a maximal rank 2 elementary abelian subgroup 
such that $CZ \not \in \cB (G)$. Note that these maps are surjective maps with kernel giving 
the isomorphism $\widetilde H^0 (\cD_G ^*; F)\cong \widetilde H^0 (\cA_{\geq 2} (G); A)$.  
So removing these does not change the cohomology groups of the chain complex at dimension 1.
 
What is left in $\cE_1''$ are the $G$-conjugacy classes of subgroups $C$ such that $E=CZ$ belongs 
to the big component $\cB(G)$, but $C\not \leq H$. These are the subgroups considered in Case 2 
in the proof of Theorem \ref{thm:Rhetorical}. Note that the only subgroups left in $\cE_2''$ are 
the ones that belong to the big component $\cB(G)$ and satisfy $E\not \leq H$. Recall that for 
each $C\in \cE_1''$ there is a chain complex of the form 
\[
\xymatrix{0 \ar[r]  & F(\overline C_G(C)) \ar[r] & \bigoplus \limits _{\overline D \in \mathcal{T}_1} 
F( C_{\overline C_G(C)} (\overline D) / \overline D   ) & \\  & \ar[r] &  \bigoplus \limits 
_{\overline E \in \mathcal{T}_2} \hom _{\aut_{\overline C_G(C)} (\overline E )} (\St_{\overline E} , 
F( C_{\overline C_G(C)} (\overline E) / \overline E   ) )  \ar[r] & \cdots ,}
\]
which is the chain complex $C^*_{\St} (F)$ for the group $\overline C_G(C)=C_G(C)/C$. By taking 
the direct sum of these chain complexes over all $C\in \cE_2'$, we obtain the chain complex given 
above in low dimensions with the same maps as $\bar \delta ^0$ and $\bar \delta ^1$.
From this we can conclude that 
$$H^1 (\cD_G^*; F) \cong \bigoplus \limits _{C \in \cE_1''} \widetilde H^0 (\cD^*_{\overline C_G(C)} ; F) 
\cong \bigoplus \limits _{C \in \cE_1''} \widetilde H^0 (\cA_{\geq 2} \bigl ( \overline C_G(C) \bigr ) /\overline C_G(C); A).$$
Note that the group $\overline C_G(C)$ has $\overline Z=CZ/C$ as a central subgroup, and the poset 
$\cA_{\geq 2} ( \overline C_G(C) ) $ is equivariantly homotopy equivalent to the poset of elementary 
abelian subgroups $\overline E$ of $\overline C_G(C)$ such that $\overline Z < \overline E$. 
Let $X_C$ denote the poset of elementary abelian subgroups $E\in \cA_{\geq 2} (G)$ such that $E>CZ$. 
We have $$\widetilde H^0 (\cA_{\geq 2} \bigl ( \overline C_G(C) \bigr ) /\overline C_G(C); A) \cong \widetilde 
H^0 (X_C/ C_G(C) ; A).$$ Also note that there is a 1-1 correspondence between $G$-conjugacy classes 
of $C \not  \leq H$ and $G$-conjugacy classes of subgroups $E=CZ$ with $E \not \leq H$.  This is because 
if $C,C' \leq CZ$ then $C$ and $C'$ are $G$-conjugate (note that if $K=\langle a, z \rangle$, then  
$[c,a]=cac^{-1}a^{-1}=z$, so $aca^{-1}= z^{-1}c$).  We conclude that 
$$H^1 (\cD_G^*; F)  \cong  \bigoplus \limits _{E\in \cE_m} \widetilde H^0 ( X_C/ C_G(C); A),$$ 
where $\cE_m$ is the set of $G$-conjugacy class representatives of elementary abelian subgroups 
of the form $CZ$ such that $C\not \leq H$ and $CZ$ belongs to the big component of $\cA_{\geq 2} (G)$. 
In fact, if it does not belong to the big component, then $CZ$ is maximal and hence we would have 
$X_C=\emptyset$, so we can actually drop this assumption for the above sum and take $\cE_m$ 
as the set of $G$-conjugacy class representatives of rank 2 elementary abelian subgroups $E$ 
such that $Z\leq E$ and $E \not \leq H$.
 
Let $X=\cA_{\geq 2} (G)$ and $Y\subseteq X$ denote the subposet consisting of all $E\in X$ such 
that either $E\leq H$ or $\rk (EZ)\geq 3$.  It is easy to see that $X-Y$ is a discrete set of rank 2 
elementary abelian subgroups $E$ such that $Z\leq E$ and $E \not \leq H$.  Note that the poset 
$Y$ is equivariantly contractible to the point $\{K\}$, so we can now apply an equivariant version 
of \cite[Lemma 6.2]{BoucMazza-Extra} to obtain 
$$H^1 (X/G ; A) \cong \bigoplus \limits _{E \in \cE_m } \widetilde H^0 (X_C / C_G(C) ; A).$$
Together with the earlier isomorphism, we conclude that $H^1 (\cD_G; F)$ is isomorphic 
to $H^1 (X/G; A).$
\end{proof}

Now we are ready to prove Theorem \ref{thm:Main3} stated in the introduction.

\begin{thm}\label{thm:ObsD} Let $G$ be a $p$-group with $p$ odd. Let $D$ denote the $p$-biset 
functor of Dade group. Then there is an exact sequence  of abelian groups
$$ 0 \to \obs(D(G)) \to H^1 (\cA_{\geq 2} (G) /G ; \Zz) \to H^1 (\cA_{\geq 2} (G)/G; \Zz/2)$$
where the second map is induced by the mod $2$ reduction map $\Zz \to \Zz/2$.
\end{thm}

\begin{proof} By  Proposition \ref{pro:ObsCb}, $\obs(D(G)) \cong H^1 (\cD_G ^* ; C_b)$. Consider 
the long exact cohomology sequence associated to the short exact sequence $0 \to C_b \to 
R_{\Qq}^* \to D_t ^{\Omega} \to 0$. Note that since $p$ is odd we have $D_t ^{\Omega} =D_t$. 
This gives an exact sequence of the form
$$\cdots \to \obs (R_{\Qq}^* (G))  \maprt{\pi^*} \obs (D_t (G)) \to H^1 (\cD^*_G ; C_b)  \to 
H^1 (\cD^*_G ;  R_{\Qq} ^* ) \to H^1 (\cD_G ^* ; D_t ) \to \cdots $$
We have shown earlier that the map $\pi^*$ is surjective (see Remark \ref{rem:Surjective}). 
Hence we obtain a short exact sequence of abelian groups of the form
$$ 0 \to \obs (D(G))  \to H^1 (\cD^*_G ;  R_{\Qq} ^* ) \to H^1 (\cD_G ^* ; D_t ).$$ 
Note that $\partial R_{\Qq} ^* (H)\cong \Zz$ and $\partial D_t (H) \cong \Zz/2$ for every 
$H$ that is cyclic or isomorphic to a generalized quaternion group.  Hence the proof follows 
from Proposition \ref{pro:H1Calc}.
\end{proof}
 
We conclude with the following.
 
\begin{pro} Let $G$ be a $p$-group with $p$ odd. If $D_f$ denotes the quotient functor $D/D_t$, 
then $$\obs (D_f (G))\cong H^1 (\cA _{\geq 2} (G) /G ; \Zz ).$$ 
\end{pro}

\begin{proof} The exact sequence $0 \to R_{\Qq} ^* \to B^* \to D/ D_t \to 0$ gives a long exact sequence
$$\cdots  \to \obs( B^*(G) ) \to \obs( D_f(G) )  \maprt{\delta} H^1 (\cD^*_G ;  R_{\Qq} ^*) \to H^1 (\cD_G ^* ; B^* ) 
\to \cdots .$$ By \cite[Theorem 3.1]{Coskun-GlueBS}, we have $\obs(B^* (G)) = 0$, and by Lemma \ref{lem:ExtCoh} 
and \ref{lem:Vanish}, we have 
\[
H^1 (\cD_G^*; B^*)\cong \ext^1_{R\cD_G} (J,B^*)=0,
\]
hence we obtain $\obs(D_f(G))\cong H^1 (\cD_G ^* ; R_{\Qq} ^* ) $. Now the result follows from Proposition 
\ref{pro:H1Calc} and from the fact that $\partial R_{\Qq} ^* (H)\cong \Zz$ for every $H$ that is cyclic or 
isomorphic to a generalized quaternion group. 
\end{proof}

\end{document}